\documentclass[11pt]{amsart}
\usepackage{hyperref}
\usepackage{amsfonts,amssymb,amsmath,amsthm,amscd,euscript,array,mathrsfs}
\usepackage[all]{xy}
\usepackage{bbm}
\usepackage{booktabs}
\usepackage{color}

\usepackage[foot]{amsaddr}
\input{mymacros-4.sty} 

\title[Capelli operators  and the Dougall-Ramanujan identity]{Capelli operators for spherical superharmonics and the Dougall-Ramanujan identity}

\author{Siddhartha Sahi $^{\mathrm{\lowercase{a}}}$}
\address{ $^{\mathrm{\lowercase{a}}}$Department of Mathematics,    
    Rutgers University,
    110 Frelinghuysen Rd, 
    Piscataway, NJ 08854-8019,
    USA
  }
  \email{sahi@math.rutgers.edu}

  \author{Hadi Salmasian$^{\mathrm{\lowercase{b}}}$}

  \address{
  $^{\mathrm{\lowercase{b}}}$ Department of Mathematics and Statistics,
    University of Ottawa,
    585 King Edward Ave,
    Ottawa, Ontario,
    Canada K1N 6N5%
  }

  \email{hadi.salmasian@uottawa.ca}

\author{Vera Serganova$^{\mathrm{\lowercase{c}}}$}
\address{
$^{\mathrm{\lowercase{c}}}$ Department of Mathematics,
    University of California at Berkeley,
    969 Evans Hall,
    Berkeley, CA 94720,
    USA
    } 
\email{serganov@math.berkeley.edu}

\thanks{\textsc{Acknowledgements.} The first and the second authors thank  Christoph Koutschan and Doron Zeilberger for helpful correspondences regarding Theorem~\ref{theorem-identity}.
The research of Siddhartha Sahi was partially supported by a Simons Foundation grant (509766), of Hadi Salmasian by an NSERC Discovery Grant (RGPIN-2018-04044), and of Vera Serganova by an NSF Grant  (1701532).\\[-3mm]}

\begin{document}

\begin{abstract}
Let $(V,\omega)$ be an orthosymplectic  $\Z_2$-graded vector space and  let $\g g:=\g{gosp}(V,\omega)$ denote
the Lie superalgebra of similitudes of $(V,\omega)$.
It is known that as a $\g g$-module, the space $\sP(V)$ of superpolynomials on $V$ is completely reducible, unless $\dim V_\eev$ and $\dim V_\ood$ are positive even integers and 
$\dim V_\eev\leq \dim V_\ood$. When $\sP(V)$ is \emph{not} a 
completely reducible $\g g$-module, we construct a natural basis $\{D_\la\}_{\la\in\cI}^{}$ of ``Capelli operators'' for the algebra 
$\sPD(V)^{\g g}$ of $\g g$-invariant superpolynomial superdifferential operators on $V$, where the index set $\cI$ is the set of integer partitions of length at most two. We compute  the action of the operators $\{D_\la\}_{\la\in\cI}^{}$ on maximal indecomposable components of $\sP(V)$ explicitly, in terms of Knop-Sahi interpolation polynomials. Our results show that, unlike the cases where $\sP(V)$ is completely reducible, the eigenvalues of a subfamily of the $D_\la$ are \emph{not} given by specializing the Knop-Sahi polynomials. Rather, the formulas for these eigenvalues involve suitably regularized forms of these polynomials.
This is in contrast with what occurs for previously studied Capelli operators. 
In addition,  we demonstrate a close relationship between our eigenvalue formulas for this subfamily of Capelli operators and  the Dougall-Ramanujan hypergeometric identity. 

We also transcend our results on the eigenvalues of Capelli operators to the Deligne category  
$\mathsf{Rep}(O_t)$. More precisely, we define categorical Capelli operators $\{\bsm D_{t,\la}\}_{\la\in\cI}^{}$ that induce morphisms of  indecomposable components of symmetric powers of $\V_t$, where $\V_t$ is  the generating object of $\mathsf{Rep}(O_t)$. We obtain formulas for the eigenvalue polynomials associated to the $\left\{\bsm D_{t,\la}\right\}_{\la\in\cI}$ that are analogous to our results for the operators $\{D_\la\}_{\la\in\cI}^{}$. 
\end{abstract}
\maketitle

\section{Introduction}
\label{sec:introduction}

Let $V:=V_\eev\oplus V_\ood$ be a vector superspace equipped with a non-degenerate even supersymmetric bilinear form $\omega:V\times V\to\C$, and let
$\g{osp}(V,\omega)$ denote the orthosymplectic Lie superalgebra that leaves $\omega$ invariant. 
Set \[
\g g:=
\g{gosp}(V,\omega):=
\g{osp}(V,\omega)\oplus\C z,
\] where $z$ is a central element of $\g g$.
Then  $V$ has a natural  $\g g$-module structure, where the action of $z$ on $V$ is defined to be  $-\yek_V$. 
The $\g g$-module structure of $V$ induces a  
canonical $\g g$-module structure on 
the superalgebra 
$\sP(V)$
of superpolynomials on $V$, and 
the superalgebra
 $\sD(V)$ 
 of constant-coefficient superdifferential operators on $V$. 
Indeed $\sP(V)\cong \sS(V^*)$ and $\sD(V)\cong \sS(V)$ as $\g g$-modules. When $\dim V_\ood=0$, studying $\sP(V)$  is the subject of the classical theory of spherical harmonics. For an elegant exposition of this theory we refer the reader to~\cite{GoodmanWallach}.

Set $d_i:=\dim V_i$ for $i\in\left\{\eev,\ood\right\}$.
It is known that $\sP(V)$ is a semisimple and multiplicity-free $\g g$-module unless $d_\eev,d_\ood\in 2\Z^+$ and $d_\eev\leq d_\ood$ (see~\cite{ShermanPhD,CoulembierJLie2013}). 
Let $\sPD(V)$ denote the superalgebra of superpolynomial-coefficient superdifferential operators 
 on $V$, equipped with the natural $\g g$-module structure defined by $x\cdot D:=xD-(-1)^{|D|\cdot |x|}Dx$ for homogeneous $x\in\g g$ and $D\in \sPD(V)$ (for further details see for example~\cite[Sec. 2]{SahSalAdv}). Then there is a canonical $\g g$-module isomorphism
\begin{equation}
\label{eq:sPD=sPsD}
\sPD(V):=\sP(V)\otimes\sD(V).
\end{equation}
Let $\cI$ be the set of integer partitions of length at most two, that is, 
\[
\cI:=\left\{(\la_1,\la_2)\in\Z^2\,:\,\la_1\geq\la_2\geq 0\right\}.
\]
In the cases that $\sP(V)$ is a semisimple and multiplicity-free $\g g$-module,
the irreducible components of $\sP(V)$ are naturally indexed by 
elements of $\cI$  (see 
\cite{GoodmanWallach,CoulembierJLie2013,ShermanPhD}). Then by a general algebraic construction 
(see 
the discussion at the end of this section,
or \cite{Sahi94,SahSalSer})
one obtains 
a distinguished basis 
$\left\{D_\la\right\}_{\la\in\cI}$ of \emph{Capelli operators}
for the algebra $\sPD(V)^{\g g}$ of $\g g$-invariant differential operators. By Schur's Lemma, the operators $D_\la$ act on irreducible components of $\sP(V)$ by scalars. The
problem of computing these scalars was addressed in~\cite{SahSalSer}, among several other examples. We remark that the problem of computing eigenvalues of Capelli operators (which we will refer to as the \emph{Capelli eigenvalue problem}) has a long history, and  has been studied extensively in the general context of multiplicity-free actions of 
Lie~(super)algebras~\cite{AllSahSal,Knop97,KostantSahi91,KostantSahi93,Sahi94,SahiZhang1,SahiZhang2,SahSalQ,Wallach90}. 
In all of the previously investigated instances of the Capelli eigenvalue problem, 
the formulas for the eigenvalues 
turn out to be specializations of families of interpolation polynomials, such as Knop-Sahi polynomials, Sergeev-Veselov polynomials, Okounkov interpolation polynomials, or Ivanov polynomials. For the definition and properties of these families of polynomials, we refer the reader to
\cite{KSDiff,OkounkovOlshanski,SVSuperJack,OkounkovBC,KoornwinderBC,Ivanov}. 
In particular, in~\cite[Theorem 1.13]{SahSalSer} we proved that 
the eigenvalues of the Capelli basis $\{D_\la\}_{\la\in\cI}$ on irreducible components of $\sP(V)$ are obtained from the  two-variable interpolation polynomials  previously defined by F. Knop and the first author~\cite{KSDiff} at the parameter value $\frac{1}{2}
\mathrm{sdim} V-1$, where $\mathrm{sdim}V:=\dim V_\eev-\dim V_\ood$.

In this paper, we are interested in defining the Capelli operators and computing their actions on $\sP(V)$ in the cases where $\sP(V)$ is \emph{not} a semisimple $\g g$-module. Thus, henceforth we will assume that $d_\eev=2m$ and $d_\ood=2n$ for $m,n\in\N$, where \[
k:=n-m\geq 0.
\]
Because of non-semisimplicity of $\sP(V)$, 
the usual definition of Capelli operators (see \cite{Sahi94,SahSalAdv,SahSalSer}) needs to be tweaked slightly. Furthermore,
 elements of $\sPD(V)^{\g g}$ are not necessarily diagonalizable  on $\sP(V)$, and thus we are naturally forced to consider their Jordan decompositions.

We show that in the non-semisimple case one still has a natural basis 
$\{D_\la\}_{\la\in\cI}$ of $\sPD(V)^{\g g}$, but a new phenomenon occurs in relation to their spectra: unlike the previous (semi-simple) instances of the Capelli eigenvalue problem, the eigenvalues of the Capelli basis are \emph{not} always specializations of interpolation polynomials. 
Rather, for a subfamily of this basis, one needs polynomials that are obtained from Knop-Sahi interpolation polynomials by removing their \emph{singular part}, that is, the part whose coefficients have poles. We provide two different formulas for the eigenvalues of this subfamily that are related to each other through a curious polynomial identity. We prove the latter polynomial identity  using the classical Dougall-Ramanujan hypergeometric identity.

To explain our main results, we begin with the definition of the Knop-Sahi polynomials. 
 We will only consider these polynomials in two variables. For the definition of  these polynomials in  the $n$-variable case, see~\cite{KSDiff}. 
 As usual, for $m\in\Z^{\geq 0}$ we define the \emph{falling factorial} 
$ a^{\ul m}$ 
 to be 
 \[
 a^{\ul m}:=a(a-1)\cdots (a-m+1).
 \]
Let $\Bbbk:=\Q(\kappa)$ be the field of rational functions in a parameter $\kappa$ with coefficients in $\Q$.
For  $\lambda\in\cI$, let $P_\lambda^\kappa\in\Bbbk[x,y]$ be defined by 
\begin{align}
\label{KNPsymformu}
P_\la^\kappa(x,y)&:=
\sum_{i+j\leq \la_1-\la_2}
\frac{
(\la_1-\la_2)!(\kappa+1)^{\underline{\la_1-\la_2-i}}
(\kappa+1)^{\underline{\la_1-\la_2-j}}
}
{
i!j!(\la_1-\la_2-i-j)!(\kappa+1)^{\underline{\la_1-\la_2}}
}
x^{\underline{\lambda_2+i}}
y^{\underline{\lambda_2+j}}.
\end{align}
The polynomial $P_\lambda^{\kappa}$ is symmetric  in the variables $x$ and $y$, with
leading term equal to $x^{\lambda_1}y^{\lambda_2}$. An important property of the 
polynomial $P_\la^\kappa$ is the following. 
\begin{thm}
\label{thm:KnSa}
{\rm (Knop--Sahi \cite{KSDiff})}
$P_\la^\kappa$ is the unique symmetric polynomial 
of degree less than or equal to $|\la|:=\la_1+\la_2$ 
in $\Bbbk[x,y]$
that satisfies the following conditions:
\begin{itemize}
\item[\rm(i)] $P^{\kappa}_\lambda(\mu_1-\kappa-1,\mu_2)=0$ for 
partitions $\mu\in\cI$ such that $|\mu|\leq |\lambda|$ and $\mu\neq \lambda$.
\item[\rm (ii)] 
$P^{\kappa}_\lambda(\lambda_1-\kappa-1,\lambda_2)=H_\la(\kappa)$, where
\begin{equation}
\label{eq:Hlak}
H_\la(\kappa):=(\lambda_1-\lambda_2)!\lambda_2!(\lambda_1-1-\kappa)^{\underline{\lambda_2}}.
\end{equation}

\end{itemize}
\end{thm}

For certain $\la\in\cI$, the coefficients of $P_\la^\kappa$ have poles. It is straightforward to verify that these poles are always simple and occur at $\kappa\in\Z^{\geq 0}$. 
Let us now define three types of elements of $\cI$.

\begin{dfn}
\label{dfn:reg-qreg-sing}
Let $k_\circ\in\Z^{\geq 0}$. An element $\lambda\in\cI$ is called 
\begin{itemize}
\item[--] $k_\circ$-\emph{regular}, if 
$\lambda_1\leq k_\circ$ or $\lambda_1-\lambda_2=k_\circ+1$ or $\lambda_1-\lambda_2\geq 2k_\circ+3$.
\item[--] $k_\circ$-\emph{quasiregular},
if $\lambda_1\geq k_\circ+1$ and $\lambda_1-\lambda_2\leq k_\circ$. 
\item[--] $k_\circ$-\emph{singular}, if $k_\circ+2\leq \lambda_1-\lambda_2\leq 2k_\circ+2$. 
\end{itemize}
We denote the sets of $k_\circ$-regular, $k_\circ$-quasiregular, and $k_\circ$-singular elements of $\cI$ by
$\cI_{k_\circ,\mathrm{reg}}$, $\cI_{k_\circ,\mathrm{qreg}}$, and 
$\cI_{k_\circ,\mathrm{sing}}$.

\end{dfn}

\begin{rmk}
\label{daggerinv}
Here is a  more concrete explanation of  Definition \ref{dfn:reg-qreg-sing}. 
Recall that $k:=n-m\in\Z^{\geq 0}$.
The involution $\la\mapsto\la^\dagger$ on $\Z^2$, defined by 
\begin{equation}
\label{eq:invol}
(\lambda_1,\lambda_2)\mapsto (\lambda_2+k+1,\lambda_1-k-1),
\end{equation}
yields a bijection between $k$-quasiregular  and $k$-singular partitions of the same size. For $k$-regular partitions $\la=(\la_1,\la_2)$ satisfying $\la_1-\la_2=k+1$, we have 
$\la^\dagger=\la$. For all other  $\la\in\cPr$ we have $\la^\dagger\not\in\cI$. 
\end{rmk}
 The following proposition is straightforward to verify using~\eqref{KNPsymformu}. 
\begin{prp}
\label{prp:polesequiv}
For $\la\in\cI$ and $k_\circ\in\Z^{\geq 0}$,
the following statements are equivalent.
\begin{itemize}
\item[\rm (i)] The coefficients of $P^{\kappa}_\lambda$ do not have poles at $\kappa=k_\circ$.
\item[\rm (ii)] $\la\not\in\cI_{k_\circ,\mathrm{sing}}$.
\end{itemize} 
\end{prp}

The construction of the Capelli basis of the algebra
$\sPD(V)^{\g g}$ relies on the structure of  $\sP(V)$ and $\sD(V)$ as $\g g$-modules. 
The algebras $\sP(V)$ and $\sD(V)$ are naturally graded  by degree and order respectively, so that 
\[
\sP(V)\cong \bigoplus_{d\geq 0}^\infty\sP^d(V)\quad\text{and}\quad 
\sD(V)\cong \bigoplus_{d\geq 0}^\infty\sD^d(V).
\]
From now on we set \begin{equation}
\label{eqP'kk}
\cI'_k:=\cPr\cup\cPq.
\end{equation}
The indecomposable components of $\sP^d(V)$ can be indexed naturally by partitions $\la\in\cI'_k$ such that $|\la|=d$ (see 
Proposition~\ref{prp:decompositionofS(V)} below).
That is,  
\begin{equation}
\label{eq:Pdco}
\sP^d(V)\cong\bigoplus_{\la\in_d^{}\cI'_k}V_\la,
\end{equation}
where $_d^{}\cI'_k:=\left\{\la\in\cI'_k\,:\,|\la|=d\right\}$, and each $V_\la$ is an indecomposable $\g g$-module. 
Furthermore, the canonical non-degenerate pairing $\sD^d(V)\otimes \sP^d(V)\to\C$ yields a $\g g$-module isomorphism
\begin{equation}
\label{eq:Dddeco}
\sD^d(V)\cong \sP^d(V)^*\cong 
\bigoplus_{\la\in _d^{}\cI'_k}V_\la^*.
\end{equation}
From Proposition~\ref{prp:decompositionofS(V)} it follows that if $\la\neq\mu$, then   $V_\la$ and $V_\mu$ have disjoint irreducible composition factors.
Thus from~\eqref{eq:Pdco} and~\eqref{eq:Dddeco}
we obtain 
\begin{align}
\label{eq:sPD=VmuVla}
\sPD(V)^\g g
\cong
\bigoplus_{\la,\mu\in\cI'_k}(V_\la\otimes V_\mu^*)^{\g g}
\cong
\bigoplus_{\la,\mu\in\cI'_k}
\Hom_{\g g}(V_\mu,V_\la)
\cong
\bigoplus_{\la\in\cI'_k}\Hom_{\g g}(V_\la,V_\la).
\end{align}
Proposition \ref{prp:decompositionofS(V)} also implies that 
\begin{equation}
\label{eq:dimHom}
\dim
\Hom_{\mathfrak g}(V_\lambda,V_\la)
=\begin{cases}
1& \text{ if $\lambda$ is $k$-regular,}\\
2&\text{ if $\lambda$ is $k$-quasiregular}.
\end{cases}
\end{equation}
Indeed when $\la\in \cPq$, there exists a nilpotent element of  $\Hom_{\mathfrak g}(V_\lambda,V_\la)$ that factors through the isomorphism $\mathrm{cosocle}(V_\la)\cong \mathrm{socle}(V_\la)$. 
{By Corollary~\ref{cor:Casimirnilp}
the space
$\Hom_{\mathfrak g}(V_\lambda,V_\la)$ has a natural direct sum decomposition into two one-dimensional subspaces, that is,
\begin{equation}
\label{eq:dimHom2}
\Hom_{\mathfrak g}(V_\lambda,V_\la)\cong
\C\yek_{V_\la}\oplus \C N_\la,
\end{equation}
where 
$N_\la$ is the nilpotent part of the  Jordan decomposition of 
$C\big|_{V_\la}$, with $C$ denoting the Casimir operator of $\g g$ (note that $N_\la^2=0$).
We now use~\eqref{eq:sPD=VmuVla} 
and~\eqref{eq:dimHom2} to define the family $\{D_\la\}_{\la\in\cI}^{}$.

\begin{dfn}
For $\la\in \cI$, we define $D_\la\in\sPD(V)^{\g g}$ as follows.
\begin{equation*}
D_\la\sim\begin{cases}
\yek_{V_\la}\in\Hom_{\g g}(V_\la,V_\la)&\text{ if }
\la\in\cI'_k\\
N_\la&\text{ if }\la\in\cPs.
\end{cases}
\end{equation*}
Here ``$\sim$'' means $D_\la$ is the element of $\sPD(V)^{\g g}$ that corresponds to either $\yek_\la$ or $N_\la$ via  the isomorphism ~\eqref{eq:sPD=VmuVla}. 
The operators $D_\la\in\sPD(V)^{\g g}$, where $\la\in\cI$,  are called the 
\emph{Capelli operators}.  

\end{dfn}
}
From~\eqref{eq:sPD=VmuVla}
it is evident that the family  $\{D_\la\}_{\la\in\cI}$ is a basis of $\sPD(V)^{\g g}$.

\section{Main results}

Now let $\la\in\cI$ and let $\mu\in\cI'_k$. 
Then by Schur's Lemma $D_\la(V_\mu)\sseq V_\mu$, and therefore  the restriction  $D_\la\big|_{V_\mu}\in\Hom_{\g g}(V_\mu,V_\mu)$ can be expressed as
\begin{equation}
\label{dlamuevonV}
D_\la\big|_{V_\mu}=d_{\la,\mu}\yek_{V_\mu}^{}+d'_{\la,\mu}N_{\mu},
\end{equation}
where 
$d_{\la,\mu},d'_{\la,\mu}\in\C$ (note that $N_\mu=0$ for $\mu\in\cPr$).
Our main results in this paper address the problem of computing formulas for
$d_{\la,\mu}$ and $d'_{\la,\mu}$. 
From Proposition~\ref{prp:CandEgen} it follows that there exists a  symmetric polynomial $f_\la\in\C[x,y]$ of degree $|\la|:=\la_1+\la_2$  such that
\[
d_{\la,\mu}=f_\la(\mu_1-k-1,\mu_2)\text{ for all }\mu\in\cI'_k.
\]  
We call $f_\la$ the \emph{eigenvalue polynomial} of $D_\la$ (see Definition \ref{dfn-eigenvalpolyn}). 
It turns out (see Proposition~\ref{lem:D=Sf+Dsqf}) that 
\[
D_\la\big|_{V_\mu}=f_\la(\mu_1-k-1,\mu_2)\yek_{V_\la}+\square f_\la(\mu_1-k-1,\mu_2) N_\la,
\]
where $f\mapsto\square f$ is the differential operator defined by
\begin{equation}
\label{eq:squareoff}
\square f(x,y):=\frac{1}{4(x-y)}
\left(
\frac{\partial f}{\partial x}(x,y)-\frac{\partial f}{\partial y}(x,y)
\right)
.
\end{equation}
Thus, both $d_{\la,\mu}$ and $d'_{\la,\mu}$ are uniquely determined by $f_\la$. 
The problem of computing $f_\la$ is solved by
Theorems~\ref{theorem1}--\ref{theorem3} below.
 \begin{thmx}
\label{theorem1}
Let $\la\in\cPr$. Then \[
f_\la=
\frac{1}{H_\la(k)}P_\la^k
,
\]
where $H_\la(\kappa)$ is defined in \eqref{eq:Hlak}.
\end{thmx}

\begin{thmx}
\label{theorem2}
Let $\la\in\cPs$. Then 
\[
f_\la
=\frac{4(\la_1-\la_2-k-1)}{H'_{\la^\dagger}(k)}P_{\la^\dagger}^k,
\] where $H'_{\la^\dagger}(k)$ denotes the derivative of $H_{\la^\dagger}(\kappa)$ at $\kappa=k$. 
\end{thmx}
The formulas for $f_\la$ in Theorems~\ref{theorem1}--\ref{theorem2} 
still follow the pattern of specializing interpolation polynomials. The new phenomenon that was described in
Section~\ref{sec:introduction}
occurs for the formulas of $f_\la$ 
when $\la\in\cPq$.
\begin{thmx}
\label{theorem4}
Let $\la\in\cPq$. Then
\begin{equation}
\label{eq:tm3lim}
\displaystyle f_\la=
\lim_{\kappa\to k}
\left(\frac{P_\la^\kappa}{H_\la(\kappa)}
+\frac{P_{\la^\dagger}^\kappa}{H_{\la^\dagger}(\kappa)}
\right)
.\end{equation}
\end{thmx}

\begin{rmk}
Note that 
both $\frac{P_\la^\kappa}{H_\la(\kappa)}$ and 
$\frac{P_{\la^\dagger}^\kappa}{H_{\la^\dagger}(\kappa)}$ have poles at $\kappa=k$ 
(indeed $H_{\la}(k)=0$), but 
the poles on the right hand side of 
\eqref{eq:tm3lim} cancel out and the limit is well defined. 
\end{rmk}

Since the leading term of $P_\la$ is $x^{\la_1}y^{\la_2}$, 
the polynomials $\{P_\la^\kappa\}_{\la\in\cI}^{}$ form a basis of the algebra $\Bbbk[x,y]^{S_2}$ of symmetric polynomials in $x$ and $y$ with coefficients in $\Bbbk$. Indeed for any $k_\circ\in\C$ such that $k_\circ\not\in\Z^{\geq 0}$, the polynomials $\{P_\la^{k_\circ}\}_{\la\in\cI}^{}$ form a basis of $\C[x,y]^{S_2}$. However, we cannot set $\kappa:=k_\circ$  when
$k_\circ\in\Z^{\geq 0}$, 
because the coefficients of the $P_\la^\kappa$ can have poles at $\kappa=k_\circ$. In this case, one can still obtain a natural basis of $\C[x,y]^{S_2}$ by first suitably separating the regular part of $P_\la^\kappa$ and then setting $\kappa:=k_\circ$.  We will describe this process more precisely below.

\begin{dfn}
Let $f(x,y)\in\Bbbk[x,y]$ and let 
$k_\circ\in\Q$. Assume that 
the coefficients of $(\kappa-k_\circ)f(x,y)$ do not have any poles at $\kappa=k_\circ$. 
\begin{itemize}
\item[(i)]
The \emph{singular part} of $f(x,y)$ 
at $\kappa=k_\circ$ 
is 
the polynomial 
$\mathrm{Sing}_{k_\circ}(f)\in\Q[x,y]$ defined by 
\[
\mathrm{Sing}_{k_\circ}(f;x,y):=\lim_{\kappa\to k_\circ}
(\kappa-k_\circ)f(x,y).
\]
 
\item[(ii)]
The \emph{regular part} of $f(x,y)$ at $\kappa=k_\circ$ is 
the polynomial 
$\mathrm{Reg}_{k_\circ}(f)\in\Q[x,y]$ defined by
\[
\mathrm{Reg}_{k_\circ}(f;x,y):=\lim_{\kappa\to k_\circ}
\left(f(x,y)-\frac{1}{\kappa-k_\circ}\mathrm{Sing}_{k_\circ}(f;x,y)\right).
\]
\end{itemize} 
\end{dfn}
\begin{ex}
Assume that $k_\circ=1$ and  $f(x,y)=x^2+y^2+\frac{2\kappa}{\kappa-1}xy$. Then 
$\mathrm{Sing}_{k_\circ}(f;x,y)=2xy$ and 
$\mathrm{Reg}_{k_\circ}(f;x,y)=x^2+y^2+2xy$.
\end{ex}
For $\la\in\cI$ and $k_\circ\in\C$ set \[
R^{(k_\circ)}_\la:=\mathrm{Reg}_{k_\circ}(
P_\la^\kappa).
\]

\begin{rmk}
Note that 
by Proposition~\ref{prp:polesequiv}, if $k_\circ\not\in \Z^{\geq 0}$ then 
for all $\la\in\cI$ we have 
\begin{equation}
\label{eq:Rkafhls}
R_\la^{(k_\circ)}=
\lim_{\kappa\to k_\circ}P_\la^\kappa
=P_\la^{k_\circ}.
\end{equation}
If $k_\circ\in \Z^{\geq 0}$
then \eqref{eq:Rkafhls} holds whenever
$\la\not\in\cI_{k_\circ,\mathrm{sing}}$.
\end{rmk}

The following proposition is a  straightforward consequence of the above discussion.

\begin{prp}
\label{prp:Rk00}
For $k_\circ\in\C$, the family $\left\{R^{(k_\circ)}_\la\right\}_{\la\in\cI}$ is a basis of the algebra $\C[x,y]^{S_2}$ of symmetric polynomials in the variables $x,y$. 
\end{prp}

By analogy with the completely reducible cases, Proposition~\ref{prp:Rk00} leads to the following natural question. \\[2mm]
\noindent\textbf{Problem.} Determine the coefficients 
 $M_{\la,\mu}\in\C$ such that
$
f_\la=\sum_{\mu\in\cI} M_{\la,\mu}R_\mu^{(k)}
$
for $\la\in\cI$. \\[-2mm]

Clearly Theorems~\ref{theorem1}--\ref{theorem2} {
answer this problem 
when $\la\not\in\cPq$.
Surprisingly,  in the case
$\la\in\cPq$ the formulas for the coefficients $M_{\la,\mu}$ become much more complicated. 
Before we state the result
(Theorem~\ref{theorem3} below),
we need to introduce some notation. 
For $d\geq 0$ set 
\begin{equation}
\label{df:id(d)}
\cI(d):=\{\la\in\cI\,:\,|\la|\leq d\}.
\end{equation}
For $\la\in\cPq$ set 
\begin{equation}
\label{eq:llamb}
\ell_\la:= 
\la_2-\la_1+k,
\end{equation}
so that 
$0\leq \ell_\la\leq k$, and if $\mu\in\cI(k-\ell_\la)$ 
then set
\[
\nu(\la,\mu):=(\la_1-\mu_1,\la_2+\mu_2).
\]
Note that $\nu(\la,\mu)\in\cPq$, and in particular
\[R^{(k)}_{\nu(\la,\mu)}=P^{k}_{\nu(\la,\mu)}.\]

\begin{thmx}

\label{theorem3}

Let $\la\in\cPq$. Then
\begin{align}
\label{eq:flaformula-sing}
f_{\la}&=
\frac{(\ell_\la+1)!}{(\la_1-k-1)!(\la_1+\ell_\la-k)!}
\left(
\frac{1}{(2k+2-\la_1+\la_2)!}
R_{\la^\dagger}^{(k)}
+
\sum_{\mu\in\cI(k-\ell_\la)}
M_{\la,\mu}R^{(k)}_{\nu(\la,\mu)}
\right),
\end{align}
where the $M_{\la,\mu}$ are defined by
\[
M_{\la,\mu}:=
\displaystyle \frac{
(-1)^{\ell_\la+\mu_1+\mu_2}{\mu_1\choose \mu_2}(\ell_\la+\mu_1)!
}
{
(k-\ell_\la-\mu_1-\mu_2)!\ell_\la! (\ell_\la+\mu_2+1)!(\ell_\la+\mu_1+\mu_2)!\mu_1
}
\quad\text{ if }|\mu|>0,
\]
and
\[M_{\la,(0,0)}:=\displaystyle
\frac{(-1)^{\ell_\la+1}}
{(k-\ell_\la)!(\ell_\la+1)!^2}
\left(
1-\sum_{j=\la_1-k}^{\la_1+\ell_\la-k}\frac{\ell_\la+1}{j}
\right).
\]

\end{thmx}

\begin{rmk}
\label{rmk:DSfunctor} For fixed $\la,\mu\in\cI$, the formulas for the eigenvalue of $D_\la\big|_{V_\mu}$ 
given in Theorems~\ref{theorem1}--\ref{theorem3} depend only on 
$k=n-m$ (rather than on $m$ and $n$).  
This observation has a conceptual explanation based on the Duflo-Serganova functor~\cite{DufloSerganova,SerganovaFunctor}.
We briefly recall the definition of this functor. 
Given any Lie superalgebra  $\g g$ and an element $x\in\g g_\ood$  such that $[x,x]=0$, 
 we set
$\DS_x(M):=M^x/xM$ for every $\g g$-module $M$, where 
$M^x:=\ker(x|_M^{})$ and 
$xM:=\im(x|_M^{})$.
Then $\DS_x(M)$ is a $\g g_x$-module, where 
$\g g_x:=\ker(\ad_x)/\im(\ad_x)$.
Further, 
for every $\g g$-module homomorphism $h:M\to N$ we set 
$
\DS_x(h):\DS_x(M)\to \DS_x(N)
$ 
to be the naturally induced $\g g_x$-module homomorphism.
As shown in~\cite{DufloSerganova,SerganovaFunctor}, the above assignments yield  a symmetric monoidal functor
\[
\DS_x:\mathrm{Rep}(\g g)\to \mathrm{Rep}(\g g_x).
\] 
If $\g g\cong\g{gosp}(V,\omega)$, then $\g g_x\cong \g{gosp}(V',\omega')$ where  $\sdim V'=\sdim V=-2k$. Furthermore, $\DS_x$ maps the Casimir operator of $\g g$, which we can consider as an element of $\hom_{\g g}(\C,\sS^2(\g g))$, to the Casimir operator of $\g g_x$. 
If 
$\sS^d(V)\cong\oplus I_t$ where each $I_t$ is a generalized eigenspace of the Casimir operator of $\g g$ with eigenvalue $t$, then
        $\sS^d(V_x)\cong \oplus I'_t$, where $I'_t\cong\DS_x(I_t)$ is the generalized eigenspace of the Casimir operator of $\g g_x$ with eigenvalue $t$.
One can then show that $\DS_x$ maps Capelli operators to Capelli operators and preserves their eigenspaces.        These facts imply that the eigenvalues of $D_\la\big|_{V_\mu}$ should only depend on $k$.

\end{rmk}

The proof of  Theorem~\ref{theorem3} is substantially more difficult than those of Theorems~\ref{theorem1}--\ref{theorem4}. It relies on the following identity (in the parameter $x$) which, to the best of our knowledge, is new. 
\begin{thmx}
\label{theorem-identity}
For non-negative integers $i,j,N$ such that $i+j\leq N$, 
\begin{align}
\label{eq:huge-identity}
\frac{d}{d x}
&\left(
\frac{x^{\ul{N-i}}x^{\ul{N-j}}}{x^{\ul{N}}}
\right)\\
&=
\sum_{q=0}^{j}
\
\sum_{p=i+j-q}^{\min\{N-q,N-1\}}
\frac{
(-1)^{N+p+q+1}
(N-p)^{\ul{q}}\,i^{\ul{q}}\,j^{\ul{q}}\,
(N-i-j)^{\ul{N-p-q}}\,
x^{\ul{p-i}}\,x^{\ul{p-j}}(x-p+q)
}{
(N-p)\,q!\,x^{\ul{p+1}}
(x-N+q)^{\ul{q}}
}.\notag
\end{align}
\end{thmx}
We remark that in the
special case $j=0$, Theorem~\ref{theorem-identity} is equivalent to the formula
\begin{equation}
\label{eq:j=0i}
\frac{\partial}{\partial x}(x^{\ul{N}})=\sum_{k=1}^N \frac{(-1)^{t+1}}{t}{N^{\ul{t}}}x^{\ul{N-t}},
\end{equation}
which can be proved by logarithmic differentiation of the binomial series for $(1+z)^x$. However, we are unable to find a similar quick argument for the general case. 

Our proof of  Theorem~\ref{theorem-identity} involves subtle computations that reduce it to a classical hypergeometric identity, usually referred to as Dougall's Theorem. 
Recall that a \emph{generalized hypergeometric function} is a series of the form
\begin{equation}
\label{eq:pFFq}
_p\mathbf{F}_q
\left(
\begin{matrix}a_1,\ldots,a_p\\
b_1,\ldots,b_q
\end{matrix}
\,;\,z
\right)
:=\sum_{n=0}^\infty
\frac{
a_1^{\oline{n}}
\cdots
a_p^{\oline{n}}
}{
b_1^{\oline{n}}
\cdots
b_q^{\oline{n}}
n!
}
z^n,
\end{equation}
where as usual  
\[
a^{\oline{n}}:=a(a+1)\cdots(a+n-1)\quad\text{for $n\in\N$ and}\ \ 
a^{\oline{0}}=1.
\]
Dougall's theorem states that for 
$\mathsf{a},\mathsf{b},\mathsf{c},\mathsf{d}\in\C$ such that
$\Re(\mathsf{a}+\mathsf{b}+\mathsf{c}+\mathsf{d}+1)>0$, we have 
\begin{align}
\label{eq:Dougall's}
_5\mathbf F_4
\left(
\begin{matrix}
\frac{1}{2}\mathsf{a}+1,\mathsf{a},-\mathsf{b},-\mathsf{c},-\mathsf{d}\\
\frac{1}{2}\mathsf{a},\mathsf{a}+\mathsf{b}+1,
\mathsf{a}+\mathsf{c}+1,
\mathsf{a}+\mathsf{d}+1
\end{matrix}\,
;\,1
\right)
&\\
&\hspace{-4cm}
=
\frac{
\Gamma(\mathsf{a}+\mathsf{b}+1)
\Gamma(\mathsf{a}+\mathsf{c}+1)
\Gamma(\mathsf{a}+\mathsf{d}+1)
\Gamma(\mathsf{a}+\mathsf{b}+\mathsf{c}+\mathsf{d}+1)
}{
\Gamma(\mathsf{a}+1)
\Gamma(\mathsf{a}+\mathsf{b}+\mathsf{c}+1)
\Gamma(\mathsf{a}+\mathsf{b}+\mathsf{d}+1)
\Gamma(\mathsf{a}+\mathsf{c}+\mathsf{d}+1)
}.
\notag
\end{align}
Identity \eqref{eq:Dougall's} is  a limit case of another identity for $_7\mathbf F_6$ that was discovered by Dougall (1907) and Ramanujan (1910). For the proof and further historical remarks on Dougall's Theorem, 
we refer the reader to~\cite[Sec. 2.2]{AndrewsAskeyRoy}.

Theorems~\ref{theorem1}, \ref{theorem2}, and~\ref{theorem3} were conjectured using computations that were implemented  by  \textsf{SageMath}. 
Our efforts to prove 
Theorem~\ref{theorem3}  lead us to  Theorems~\ref{theorem4} and~\ref{theorem-identity}. 

\subsection*{Capelli operators in the Deligne category $\mathsf{Rep}(O_t)$}
Recall from Remark~\ref{rmk:DSfunctor} that 
existence of certain monoidal functors between  (rigid symmetric monoidal) 
categories of modules implies that 
the formulas for $f_\la$ should only depend on the superdimension of $V$. 
Indeed it is possible to transcend the construction of the Capelli basis $\{D_\la\}_{\la\in\cI}$ to a
universal categorical framework
where the superdimension can be any complex number!  
More precisely,  in Section~\ref{sec:Deligne}
we show that we can define Capelli operators in  
the inductive completion of the Deligne category 
$\Rep(O_t)$, where $t\in\C$. Then we prove analogues of 
Theorems~\ref{theorem1}--\ref{theorem4} for the corresponding eigenvalue polynomials.

{
The definition of the categorical Capelli operators 
$\left\{
\bsm D_{t,\la}\right\}_{\la\in\cI}$
in Section~\ref{sec:Deligne} goes as follows.
The category $\Rep(O_t)$ is the  Karoubian rigid symmetric monoidal category generated by the self-dual object $\V_t$ of categorical dimension $t\in\C$. 
We introduce an algebra object 
$
\sfPD_{\V_t}
$ 
in the inductive completion of this category
with a natural action 
\[
\sfPD_{\V_t}\otimes \sfP_{\V_t}\to \sfP_{\V_t},
\]
where
$\sfP_{\V_t}:=\bigoplus_{d\geq 0}\sS^d(\V_t)$.
The algebra object $\sfPD_{\V_t}$ is the categorical analogue of $\sPD(V)$. 
Moreover, $\Hom(\yeksf,\sfPD_{\V_t})$ can be equipped with a canonical algebra structure, and the natural action 
of $\sfPD_{\V_t}$ on $\sfP_{\V_t}$ yields  a  homomorphism 
of algebras $\Hom(\yeksf,\sfPD_{\V_t})\to \End(\sfP_{\V_t})$. 
The categorical Capelli operators 
$\bsm D_{t,\la}$ that we will define in Section~\ref{sec:Deligne}
are elements of the algebra $\Hom(\yeksf,\sfPD_{\V_t})$. 
To define these operators, first we  prove that the indecomposable summands of $\sfP_{\V_t}$ are naturally indexed by elements of $\cI$ if $t\not\in 2\Z^{\leq 0}$, and by elements of 
$\cI'_{\ul k}$ if $t\in 2\Z^{\leq 0}$, where 
\begin{equation}
\label{eq:kbarrr}
\ul k:=-\frac{t}{2}
\quad\text{and}\quad
\cI'_{\ul k}:=\cI_{\ul k,\mathrm{reg}}\cup\cI_{\ul k,\mathrm{qreg}}.
\end{equation}
When $t\not\in 2\Z^{\leq 0}$, for every $\la\in\cI$ the operator $\bsm D_{t,\la}$ corresponds to the co-evaluation morphism
\[
\yeksf\xrightarrow{\ \bsm\epsilon_{\V_{t,\la}}^{}\ } \V_{t,\la}\otimes \V_{t,\la}^*.
\]
When $t\in 2\Z^{\leq 0}$, the definition of $\bsm D_{t,\la}$ is still the same for $\ul k$-regular and $\ul k$-quasiregular $\la$, but 
for $\ul k$-singular $\la$
the operator $\bsm D_{t,\la}$ represents the (unique up to scaling) nilpotent element in $\End(\V_{t,\la^\dagger})$. 
See equation~\eqref{eq:CategCap2} for further  details. 
}

After defining the operator $\bsm D_{t,\la} $, we can  consider 
its restriction
to each indecomposable component $\V_{t,\mu}$ of $\sfP_{\V_t}$ that is indexed by $\mu$.
This yields an element of the algebra 
$\End(\V_{t,\mu})$, of the form $\ul d_{\la,\mu}\yeksf +\mathsf n_{\la,\mu}$ where 
$\ul d_{\la,\mu}\in\C$ and 
$\mathsf n_{\la,\mu}^2=0$. 
Furthermore 
\[
\ul d_{\la,\mu}=\bsm f_\la(\mu_1-\ul k-1,\mu_2),
\] where $\bsm f_\la\in\C[x,y]$ is a symmetric polynomial of degree  $|\la|$. (We remark that the coefficients of $\bsm f_\la$ depend on the value of $t\in\C$.)
Theorems~\ref{thm:A'}--\ref{thm:C'} below are the extensions of Theorems~\ref{theorem1}--\ref{theorem4} to the categorical setting of $\Rep(O_t)$.

\begin{customthm}{A$^\prime$}
\label{thm:A'}

Assume that either  $t\notin 2\mathbb Z^{\leq 0}$, or that $t\in 2\mathbb Z^{\leq 0}$ and $\lambda$ is $\ul k$-regular. Then 
  \[
  \bsm f_{\lambda}=
  \frac{1}{H_\lambda\left(\ul k\right)}
  P^{\ul k}_\lambda.
  \]
\end{customthm}From now on we set  \[
\bsm c_\lambda(t):=(\lambda_1-\lambda_2)(\lambda_1-\lambda_2+t-2)
\quad\text{for $\la\in\cI$.}
\]

\begin{customthm}{B$^\prime$}

\label{thm:B'}  
Assume that  $t\in 2\mathbb Z^{\leq 0}$ and $\lambda$ is $\ul k$-singular. Then 
  \[
  \bsm f_{\lambda}=\lim_{s\to t}
  \frac{\bsm c_{\la^\dagger}(s)-\bsm c_{\la}(s)}{H_{\lambda^\dagger}\left(-\frac{s}{2}\right)}
  P^{-\frac{s}{2}}_{\lambda^\dagger}.
  \]
\end{customthm}

\begin{customthm}{C$^\prime$}
\label{thm:C'}

Assume that  $t\in 2\mathbb Z^{\leq 0}$ and $\lambda$ is $\ul k$-quasiregular. Then 
  \[
  \bsm f_{\lambda}=\lim_{s\to t}
\left(
  \frac{1}{H_\lambda\left(-\frac{s}{2}\right)}
  P^{-\frac{s}{2}}_\lambda
+
  \frac{1}{H_{\lambda^\dagger}\left(-\frac{s}{2}\right)}
  P^{-\frac{s}{2}}_{\lambda^\dagger}
\right)  .
  \]
\end{customthm}

\section{Structure of $\sP(V)$ and $\sPD(V)^{\g g}$}
\label{sec:Capelli}
Let us begin with the description of the decomposition of
the $\g g$-module $\sP(V)$ as a direct sum of indecomposable submodules. 
As will be seen in Proposition 
\ref{prp:decompositionofS(V)},
the indecomposable components of $\sP(V)$ can be characterized as generalized eigenspaces of the restriction of the Casimir operator to each homogeneous component. 
A proof of 
 Proposition \ref{prp:decompositionofS(V)}
is given  in A. Sherman's PhD thesis~\cite{ShermanPhD} 
(see also~\cite{CoulembierJLie2013}).

 Let ${\g b}^{\mathrm{st}}$ be the Borel subalgebra of 
$\g{osp}(V,\omega)$ corresponding to the
fundamental system
\[
\left\{
\varepsilon_1-\varepsilon_2, \dots,\varepsilon_{m-1}-\varepsilon_m,\varepsilon_m-\delta_1,\dots,\delta_{n-1}-\delta_n,2\delta_n
\right\},
\]
and set $\g b:={\g b}^{\mathrm{st}}\oplus\C z$. Also, let ${\g h}^{\mathrm{st}}\sseq {\g b}^{\mathrm{st}}$ denote the standard Cartan subalgebra of $\g{osp}(V,\omega)$, and set $\g h:={\g h}^{\mathrm{st}}\oplus\C z$. Let  $\zeta\in\g h^*$ be the linear functional defined by $\zeta(z)=1$ and $\zeta\big|_{{\g h}^{\mathrm{st}}}=0$.
For a ${\g b}^{\mathrm{st}}$-dominant ${\g h}^{\mathrm{st}}$-weight $\lambda$, let $V(\lambda)$ 
denote the irreducible finite dimensional $\g{osp}(V,\omega)$-module with highest weight $\lambda$. For any scalar $c\in\C$, we 
can consider $V(\lambda)$ as a $\g g$-module on which $z$ acts by $c\yek_{V(\lambda)}$. We denote the latter $\g g$-module by $V(\lambda+c\zeta)$. 

{
Recall that $C$ denotes the Casimir operator of $\g{gosp}(V,\omega)$. Then $C$ acts on $V(\lambda+c\zeta)$ by the scalar 
\[
c_\la:=
(\la,\la)+2(\la,\rho)
=
(\lambda_2-\lambda_1)(2k+2+\lambda_2-\lambda_1)
,
\] where  $
\rho:=\sum_{i=1}^m (-k-i)\varepsilon_i
+\sum_{i=1}^n(n-i+1)\delta_i$.
For $t\in \C$ let $\sP^d(V,t)$ denote the generalized $t$-eigenspace of the restriction of 
$C$ to $\sP^d(V)$.
Note that 
$c_\la=c_{\la^\dagger}$ for $\la\in\cPq$,  hence 
$\sP^d(V,c_\la)=\sP^d(V,c_{\la^\dagger})$.
For $\lambda\in\cI'_k$, set \[
V_\la:=\sP^{|\la|}(V,c_\la).
\]
The proof of the following proposition can be found in \cite[Sec. 10]{ShermanPhD}.
}

\begin{prp}
\label{prp:decompositionofS(V)}
Let $\la\in\cI'_k$.
 \begin{itemize}
\item[\rm(i)] If $\la\in\cPr$ then 
\[
V_\la\cong V\big((\la_1-\la_2)\eps_1+|\la|\zeta\big).
\]
In particular, $V_\la$ is an irreducible $\g g$-module. 
\item[\rm(ii)] If $\la\in\cPq$ then 
$V_\la$ is an indecomposable $\g g$-module
with a socle filtration of length 3.  
When $m\geq 2$, the successive
 quotients of the socle filtration of 
 $V_\la$  are isomorphic to the modules  
$V(\mu^{(i)}+|\la|\zeta)$, $1\leq i\leq 3$, where 
\[
\mu^{(1)}=\mu^{(3)}=(\la_1-\la_2)\eps_1
\ \ \text{ and }\ \ 
\mu^{(2)}=(2k+2+\la_2-\la_1)\eps_1.
\]
When $m=1$, the successive quotients of the socle filtration of $V_\la$
are isomorphic to 
\[
V(\mu^{(1)}+|\la|\zeta),\ \ V(\mu^{(2)}+|\la|\zeta)\oplus V(\mu^{(3)}+|\la|\zeta),\ \ \text{and}\ \ V(\mu^{(4)}+|\la|\zeta),
\] where
\[
\mu^{(1)}=\mu^{(4)}=(\la_1-\la_2)\eps_1,
\ 
\mu^{(2)}=(2k+2+\la_2-\la_1)\eps_1,\ \text{ and }
\
\mu^{(3)}=-\eps_1+\sum_{i=1}^{\la_1-\la_2+1}\delta_i.
\]

\end{itemize}

\end{prp}



\begin{rmk}
One significant difference between the non-semisimple and semi-simple cases is that in the non-semisimple cases  the spaces of homogeneous harmonic polynomials of any given degree are not necessarily irreducible $\g g$-modules.
However, 
for $d\leq k+1$ and $d>2k+2$
the space of harmonic polynomials of degree $d$ 
is still an irreducible $\g g$-module, isomorphic to $V( d\eps_1+d\zeta)$. 
\end{rmk}

%
From now on we 
 identify the Casimir operator $C$ with its image in $\sPD(V)^{\g g}$. Let $E\in \sPD(V)^{\g g}$ denote the degree operator (which lies in the image of the center of $\g g$).
\begin{prp}
\label{prp:CandEgen}
The operators $C$ and $E$ generate $\sPD(V)^{\g g}$. Furthermore, for any differential operator  $D\in \sPD(V)^{\g g}$ of order $d$, 
there exists a unique symmetric polynomial $f_D(x,y)$ of degree $d$
such that 
the eigenvalue of $D$ on the indecomposable constituent $V_\mu$ is equal to $f_D(\mu_1-k-1,\mu_2)$.
\end{prp}
\begin{proof}
For $\la\in\cI'_k$ set $e_\la:=\la_1+\la_2=\big((\lambda_1-(k+1)\big)+\lambda_2)+(k+1)$, and  recall that 
\[
c_\la=(\la_2-\la_1)(2k+2+\la_2-\la_1)
=
\big((\lambda_1-(k+1))-\lambda_2\big)^2-(k+1)^2.
\]
Thus $c_\la$ and $e_\la$ are symmetric polynomials in $\la_1-k-1$ and $\la_2$. 
The restriction to $V_\la$ of any operator of the form
$D:=q(C,E)$, where $q(x,y)\in\C[x,y]$, is of the form
$q(c_\lambda,e_\lambda)\yek_{V_\la}+X$, where $X\in \mathrm{End}(V_\la)$ is nilpotent.\\[1mm]
\noindent\textbf{Step 1.} 
We prove that for every 
symmetric polynomial
$h(x,y)\in\C[x,y]$ 
there exists an operator $D\in \sPD(V)^{\g g}$ of order at most $\deg h$ such that for every $\la\in\cI'_k$ the restriction $D\big|_{V_\lambda}$ is of the form
\[h(\la_1-k-1,\la_2)\yek_{V_\la}+X,
\] where $X$ is nilpotent. 
To prove this claim, we write $h$ as a polynomial in  
$\sfe_1=x+y$ and $\sfe_2=xy$, that is, 
$
h(x,y)=\sum_{i+2j\leq d}a_{i,j}\sfe_1^i\sfe_2^j$,
where $d:=\deg h$. Writing $\sfe_1$ and $\sfe_2$ in terms of $x+y+k+1$ and $(x-y)^2-(k+1)^2$, 
it follows that $h(x,y)$ can also be expressed as
\[
h(x,y)=\sum_{i+2j\leq d}b_{i,j}(x+y+k+1)^i\big((x-y)^2-(k+1)^2\big)^j,
\]
where $b_{i,j}\in\C$. 
It is easy to verify that the 
operator $D:=\sum_{i+2j\leq d}b_{i,j}E^iC^j$ satisfies the claimed properties. \\[1mm]
\noindent\textbf{Step 2.} For $d\geq 0$ set 
$\mathcal V_d:=\left\{D\in\sPD(V)^{\g g}\,:\,\mathrm{ord}(D)\leq d\right\}$, where $\mathrm{ord}(D)$ denotes the order of $D$. 
From  \eqref{eq:sPD=VmuVla} 
and \eqref{eq:dimHom}
it follows that 
$\dim \mathcal V_d=N_d:=|\cI(d)|$.
The space of symmetric polynomials of degree at most $d$ also has dimension $N_d$. 
Furthermore, operators that correspond by Step~1 to linearly independent polynomials are also linearly independent. Thus Step 1
provides $N_d$ linearly independent elements in $\mathcal V_d\cap\cA$, where $\cA$ is the subalgebra of $\sPD(V)^{\g g}$ that is generated by $C$ and $E$. This yields
$\dim \mathcal V_d\cap \cA\geq \dim\mathcal V_d$, and consequently $\mathcal V_d\sseq\mathcal \cA$. \\[1mm]
\textbf{Step 3.} Let $D\in\sPD(V)^{\g g}$ such that $\mathrm{ord}(D)=d$. By Step~2, there exists a  symmetric polynomial $f_D\in\C[x,y]$ such that $\deg f_D\leq d$ and $D$ is obtained from $f_D$ by the construction of Step~1. From Step 1 it follows that 
$d=\mathrm{ord}(D)\leq \deg f_D\leq d$. Hence $\deg f_D=\mathrm{ord}(D)=d$.  Finally, $f_D$ is unique because $\cI$ is Zariski dense in $\C^2$. 
\end{proof}

\begin{dfn}
\label{dfn-eigenvalpolyn}
For $D\in\sPD(V)^{\g g}$, the polynomial $f_D(x,y)$ whose existence is guaranteed by Proposition \ref{prp:CandEgen} will be called the \emph{eigenvalue polynomial} of $D$.
\end{dfn}

\begin{cor}
\label{cor:Casimirnilp}
For $\la\in\cPq$, the restriction of
$C$ to $V_\la$ is not diagonalizable.
In particular,  the nilpotent part of the Jordan decomposition of $C\big|_{V_\la}$ is nonzero. 
\end{cor}
\begin{proof}
Otherwise, 
Proposition~\ref{prp:CandEgen} would imply that 
$D\big|_{V_\la}$ is diagonalizable for all 
$D\in \sPD(V)^{\g g}$. In particular, 
$D_{\la^\dagger}\big|_{V_\la}$ would be  diagonalizable, which is a contradiction. 
\end{proof}


\begin{rmk}
As noted in Section~\ref{sec:introduction},
Corollary~\ref{cor:Casimirnilp} is crucial for
being able  to define the basis $\left\{D_\la\right\}_{\la\in\cI}$ of  Capelli operators for $\sPD(V)^{\g g}$. 
\end{rmk}

\section{Vanishing properties and generalized values}

\label{sec:Main}
Recall from \eqref{dlamuevonV} that $d_{\la,\mu}$ denotes the eigenvalue of $D_\la$ on $V_\mu$. The $d_{\la,\mu}$ satisfy the following vanishing properties which are deduced from elementary representation-theoretic arguments.

\begin{lem}
\label{lem:lamuvalRep}
Let $\la,\mu\in\cI$.
\begin{itemize}
\item[\rm (i)] Assume that $\la\in\cI_k'$. Then $d_{\la,\mu}=0$ for all $\mu\in\cI'_k$ such that $|\mu|\leq|\la|$ and $\mu\neq \la$. Furthermore, $d_{\la,\la}=1$.

\item[\rm (ii)] Assume that $\la\in\cPs$. 
Then $d_{\la,\mu}=0$ for all $\mu\in\cI'_k$ such that $|\mu|\leq|\la|$.

\end{itemize}
\end{lem}

\begin{proof}
From 
the isomorphism \eqref{eq:sPD=VmuVla}
it follows that
$D_\la\big|_{V_\la}=\yek_{V_\la}$ for $\la\in\cI_k'$ 
. 
For $\la\in\cPs$ we have $d_{\la,\la^\dagger}=0$
because the restriction of $D_{\la}$ to ${V_{\la^\dagger}}$ is nilpotent.
If  $|\mu|<|\la|$ then we have $D_\la\big|_{V_\mu}=0$ because $V_\mu\sseq \sP^{|\mu|}(V)$ and $\mathrm{ord}(D_\la)>|\mu|$.  
If $|\mu|=|\lambda|$, then the action of $D_\lambda$ on $V_\mu$ is obtained by restriction of the $\mathfrak g$-equivariant map 
\[
\sP(V)\otimes \sD(V)\otimes \sP(V)\to \sP(V)\ ,\ 
p\otimes D\otimes q\mapsto pDq,
\]
to a tensor product of the form $V_\eta\otimes V_\eta^*\otimes V_\mu$, where $\eta=\la$ or $\eta=\la^\dagger$ depending on wheter $\la\in\cPq$ or $\la\in\cPs$. 
As $|\mu|=|\la|$, the map $V_\eta^*\otimes V_\mu\to\mathbb C$ corresponds to  a $\g g$-invariant bilinear form 
$V_\eta^*\times V_\mu\to\mathbb C$, hence to a $\g g$-equivariant  linear map $V_\eta^*\to V_\mu^*$. 
Thus, when $V_\mu^*$ and $V_\eta^*$ do not have composition factors in common, we obtain
$D_\la\big|_{V_\mu}=0$ and in particular $d_{\la,\mu}=0$. The above facts are sufficient for verifying the claims of the lemma. 
\end{proof}
\begin{rmk}
\label{rmk:nilpJofDla}
The proof of Lemma \ref{lem:lamuvalRep} implies that if $|\mu|\leq |\la|$, then the nilpotent part of the Jordan decomposition of $D_\la\big|_{V_\mu}$ vanishes, 
unless $\la$ is $k$-singular and $\mu=\la^\dagger$.
\end{rmk}

We can now write $f_\la$ as \[
f_\la(x,y)=\sum_{i+j\leq |\la|}a_{i,j}(x^iy^j+x^jy^i),
\] and interpret the constraints $f_\la(\mu_1-k-1,\mu_2)=d_{\la,\mu}$ 
for $|\mu|\leq |\la|$
as a linear system in the coefficients $a_{i,j}$. Unfortunately,  this linear system (which a priori has
the same number of equations and variables)
does \emph{not} determine $f_\la$ uniquely
because of the redundancy that is caused by the coincidences  
\begin{equation}
\label{eq:coinc}
f_\la(\mu_1-k-1,\mu_2)=f_\la(\mu^\dagger_1-k-1,\mu_2^\dagger).
\end{equation}
But we can circumvent this issue by using the Jordan decomposition of  $D_\la\big|_{V_\mu}$ to obtain extra conditions on  $f_\la$. 


\begin{prp}
\label{lem:D=Sf+Dsqf}
Let $D\in\sPD(V)^{\g g}$ and assume that $f_D(x,y)$ is the eigenvalue polynomial of $D$. 
Then for $\la\in\cPr\cup\cPq$ we have 
\[
D\big|_{V_\la}=f_D(\la_1-k-1,\la_2)\yek_{V_\la}+\square f_D(\la_1-k-1,\la_2) N_\la,
\]
where $\square f_D$ is defined as in~\eqref{eq:squareoff}.
\end{prp}
\begin{proof}
By Proposition \ref{prp:CandEgen}
we can express $D$ as $D=p(C,E)$ for a polynomial $p(s,t)\in\C[s,t]$. Note that 
$C\big|_{V_\la}=
c_\la\yek_{V_\la}+ N_\la$ where  
$N_\la^2=0$, hence $C^d\big|_{V_\la}=c_\la^d\yek_{V_\la}+ dc_\la^{d-1}N_\la$. It follows that
\begin{equation}
\label{eq:D=pSE+dpds}
D\big|_{V_\la}=p(c_\la,e_\la)\yek_{V_\la}+
\frac{\partial p}{\partial s}
(c_\la,e_\la)
N_\la.
\end{equation}
By comparing the eigenvalues on both sides of~\eqref{eq:D=pSE+dpds} and noting that $\cI\sseq \C^2$ is Zariski dense, we obtain \begin{equation}
\label{eq:fx,y}
f_D(x,y)=p\big((x-y)^2-(k+1)^2,x+y+k+1\big).
\end{equation}
Next set 
\[
H(x,y):=\big((x-y)^2-(k+1)^2,x+y+k+1\big).
\]
Then by the chain rule we obtain 
\[
\left\{
\begin{array}{l}
\displaystyle\frac{\partial f_D}{\partial x}(x,y)
=
\frac{\partial p}{\partial s}(H(x,y))(2x-2y)+
\frac{\partial p}{\partial t}
(H(x,y)),\\[2mm]
\displaystyle
\frac{\partial f_D}{\partial y}(x,y)=
\frac{\partial p}{\partial s}(H(x,y))(2y-2x)+
\frac{\partial p}{\partial t}(H(x,y)).
\end{array}
\right.
\]
Taking the difference of the above relations yields
\begin{equation}
\label{eq:D1p-f}
\frac{\partial p}{\partial s}(H(x,y))=
\square f_D(x,y).
\end{equation}
The statement of the lemma follows from \eqref{eq:D=pSE+dpds} and \eqref{eq:D1p-f}.
\end{proof}

%
%
%
%
%
%

Using Proposition~\ref{lem:D=Sf+Dsqf}, we obtain the required extra constraints that together  with the vanishing conditions of Lemma~\ref{lem:lamuvalRep} uniquely identify the polynomials $f_\la$. 
In order to give a uniform description of all of these constraints, we use the notion of the \emph{generalized value} of  a symmetric polynomial $f(x,y)$ at $\la\in\cI$, denoted by $\widetilde{\ev}(f,\la)$, 
defined as follows.
\[
\widetilde{\ev}(f,\la):=
\begin{cases}
f(\la_1-k-1,\la_2) & \text{ if }\la\in\cPr\cup\cPq,\\
\square f(\la_1-k-1,\la_2)
& \text{ if }\la\in\cPs.
\end{cases}
\]
Then Lemma \ref{lem:lamuvalRep}
and Remark \ref{rmk:nilpJofDla} imply  the following proposition. 
\begin{prp}
\label{prp:extra-fla}
For $\la,\mu\in\cI$, if  
$|\mu|\leq |\la|$ then $\widetilde{\ev}(f_\la,\mu)=\delta_{\la,\mu}$.
%
%
%
%
%
%
%
%
%
%
%
%
%

\end{prp}

In the following corollary,  $\cI(d)$ is defined as in~\eqref{df:id(d)}.

\begin{cor}
\label{cor:ev(f,la)=tla}
Fix  a set of complex numbers $\{z_\la\,:\,\la\in\cI(d)\}$
for some $d\geq 0$. Then there exists a unique symmetric polynomial $f(x,y)$  such that $\deg f\leq d$ and 
$\widetilde{\ev}(f,\la)=z_\la$ for all $\la\in\cI(d)$.
%
\end{cor}

\begin{proof}
Follows immediately from 
 Proposition \ref{prp:extra-fla}.
\end{proof}

\section{Proofs of Theorems \ref{theorem1}, \ref{theorem2}, and \ref{theorem4}}
\label{sec:proofABD}
We now proceed towards the proofs of 
Theorems~\ref{theorem1}--\ref{theorem4}. The next lemma is a key observation.
\begin{lem}
\label{lem:squPbetal}
Let $p(\kappa;x,y)\in \Bbbk[x,y]$ and let $k_\circ\in\R$ be such that
the coefficients of $p(\kappa;x,y)$ do not have poles at $\kappa=k_\circ$.
Further, assume that for $a,b,a',b'\in \R$ we have 
\[
(a-k_\circ-1,b)=(a',b'-k_\circ-1)
 \ \text{  and }\ a-b-k_\circ-1\neq 0. 
\]
Set
\begin{equation}
\label{eq:alphabetaP}
\alpha(\kappa):=p(\kappa;a-\kappa-1,b)\ \text{ and }\ 
\beta(\kappa):=p(\kappa;a',b'-\kappa-1).
\end{equation}
Then
\[
\square p(k_\circ;a-k_\circ-1,b)=
\frac{\beta'(k_\circ)-\alpha'(k_\circ)}{4(a-b-k_\circ-1)}
\]
\end{lem}
\begin{proof}
Differentiating the equations given in \eqref{eq:alphabetaP} with respect to $\kappa$ at $\kappa=k_\circ$, we obtain
\[
\alpha'(k_\circ)=
\frac{\partial p}{\partial \kappa} (k_\circ;a-k_\circ-1,b)-
\frac{\partial p}{\partial x}(k_\circ;a-k_\circ-1,b)
\]
and 
\[
\beta'(k_\circ)=
\frac{\partial p}{\partial \kappa}(k_\circ;
a',b'-k_\circ-1)
-
\frac{\partial p}{\partial y}(k_\circ;a',b'-k_\circ-1).
\]
Taking the difference of the above relations yields the claim  of the  lemma.
\end{proof}

\begin{lem}
\label{lem:fk0squSingf}
Let $p(\kappa;x,y)\in\Bbbk[x,y]$ and let $k_\circ\in\R$ be such that the coefficients of 
$(\kappa-k_\circ)p(\kappa;x,y)$ do not have poles at $\kappa=k_\circ$.
Further, assume that for 
$a,b,a',b'\in\R$ we have 
\begin{equation}
\label{eq:a-k-1,b))}
(a-k_\circ-1,b)=(a',b'-k_\circ-1)\ \text{ and }
\ 
a-b-k_\circ-1\neq 0
.
\end{equation}
For $\kappa\in\R\bls\{k_\circ\}$ sufficiently close to $k_\circ$, set
\[
\alpha(\kappa):=p(\kappa;a-\kappa-1,b)\text{ and }
\beta(\kappa):=p(\kappa;a',b'-\kappa-1).
\]
Then 
\[
\square\mathrm{Sing}_{k_\circ}(p;a-k_\circ-1,b)=
\frac{1}{4(a-b-k_\circ-1)}\lim_{\kappa\to k_\circ}
\left(
(\beta(\kappa)-\alpha(\kappa))+(\kappa-k_\circ)(\beta'(\kappa)-\alpha'(\kappa))^{}_{}
\right)
.\] 
\end{lem}


\begin{proof}
Set $p_x:=\frac{\partial p}{\partial x}$ 
and 
$p_y:=\frac{\partial p}{\partial y}$.
Since taking the singular part commutes with partial differentiation with respect to 
$x$ and $y$, 
it suffices to prove that 
\begin{align}
\notag 
\label{eq:limpxpy=jfd}
\lim_{\kappa\to k_\circ} 
\Big(
(\kappa-k_\circ)
\big(
p_x(\kappa;a&-k_\circ-1,b)-p_y(\kappa;a-k_\circ-1,b)\big)
\Big)\\
&
=
\lim_{\kappa\to k_\circ}
\left(
(\beta(\kappa)-\alpha(\kappa))+(\kappa-k_\circ)(\beta'(\kappa)-\alpha'(\kappa))^{}_{}
\right)
.
\end{align}
Set \[
h(\kappa;x,y):=(\kappa-k_\circ)p(\kappa;x,y)\in\Bbbk[x,y].
\] Note that  $h(\kappa;x,y)$ is a smooth map 
in a neighborhood of any point of the form 
$(k_\circ,x_\circ,y_\circ)\in\R^3$.
By differentiating the relation $(\kappa-k_\circ)\alpha(\kappa)=h(\kappa;a-\kappa-1,b)$ with respect to $\kappa$ at $\kappa:=k_1$, with $k_1\neq k_\circ$ and $k_1$ sufficiently close to $k_\circ$,  we obtain
\begin{equation}
\label{eq:alpha'1}
(k_1-k_\circ)\alpha(k_1)+\alpha'(k_1)
=
\frac{\partial h}{\partial \kappa}
(k_1,a-k_1-1,b)
-
(k_1-k_\circ)
p_x
(k_1,a-k_1-1,b).
\end{equation}
Similarly, by differentiating the relation 
$(\kappa-k_\circ)\beta(\kappa)=
h(\kappa,a',b'-\kappa-1)$ with respect to $\kappa$  
we obtain 
\begin{equation}
\label{eq:beta'1}
(k_1-k_\circ)\beta(k_1)
+
\beta'(k_1)
=\frac{\partial h}{\partial\kappa} 
(k_1,a',b'-k_1-1)-
(k_1-k_\circ)
p_y
(k_1,a',b'-k_1-1).
\end{equation}
By taking the difference of \eqref{eq:alpha'1} and \eqref{eq:beta'1}
we obtain
\begin{align}
\label{Eq:k1-k0f}
\notag 
(k_1-k_\circ)
(
p_x&(k_1,a-k_1-1,b)-p_y(k_1,a',b'-k_1-1)
)\\
&=
\frac{\partial h}{\partial \kappa}
(k_1,a-k_1-1,b)-
\frac{\partial h}{\partial \kappa}
(k_1,a',b'-k_1-1)
+
\phi(k_1),\end{align}
where 
\[
\phi(k_1):=(k_1-k_\circ)\left(\beta(k_1)-\alpha(k_1)\right)
+
\left(\beta'(k_1)-\alpha'(k_1)\right)
.
\]
Note that $\lim_{k_1\to k_\circ} \phi(k_1)$ exists because \[\phi(k_1)=\frac{d}{d\kappa}
\left(
h(\kappa;a-\kappa-1)-h(\kappa;a',b'-\kappa-1)
\right)\big|_{\kappa=k_\circ},
\] and  $h(\kappa;x,y)$ is differentiable near  $(k_\circ,a-k_\circ-1,b)\in\R^3$.
Since $h(\kappa;x,y)$ is smooth in a neighborhood of the point $(k_\circ,a-k_\circ-1,b)\in\R^3$,
from \eqref{eq:a-k-1,b))} it follows that
\begin{equation}
\label{eq:limk1kcirch}
\lim_{k_1\to k_\circ}\frac{\partial h}{\partial \kappa}(k_1,a-k_1-1,b)=
\lim_{k_1\to k_\circ}\frac{\partial h}{\partial \kappa}(k_1,a',b'-k_1-1).
\end{equation}
Finally the equations \eqref{Eq:k1-k0f} and \eqref{eq:limk1kcirch} imply \eqref{eq:limpxpy=jfd}.
\end{proof}

For $\la\in\cPs$ set
\begin{equation}
\label{eq:defofrlam}
r_\la:=-\frac{H_\la(k)}{H'_{\la^\dagger}(k)}.
\end{equation}

\begin{prp}

\label{lem:sing=rlaPla}
Let $\la\in\cPs$. Then 
$\mathrm{Sing}_{k}(P_\la^\kappa)=r_{\la} P_{\la^\dagger}^k$ where 
$r_{\la}$ is defined in \eqref{eq:defofrlam}.

\end{prp}

\begin{proof}
Set $p(x,y):=\mathrm{Sing}_k(P_\la^\kappa;x,y)$. By Proposition
\ref{prp:polesequiv} we have $p(x,y)\neq 0$. 

\noindent\textbf{Step 1.} 
We prove that 
$p=sP_{\la^\dagger}^k$ for a scalar $s\neq 0$. 
To this end, by Corollary \ref{cor:ev(f,la)=tla} we need to verify that 
\[
\widetilde{\ev}(p,\mu)=0\ \ \text{ when }
|\mu|\leq |\la|\text{ and }\mu\neq\la.
\]
If $\mu$ is $k$-regular or $k$-quasiregular, then   
\begin{align*}
p(\mu_1-k_\circ-1,\mu_2)&=
\lim_{\kappa\to k_\circ} 
(\kappa-k_\circ)P_{\la}^{\kappa}(\mu_1-k_\circ-1,\mu_2)\\
&=
\lim_{\kappa\to k_\circ} 
(\kappa-k_\circ)P_{\la}^{\kappa}(\mu_1-\kappa-1,\mu_2),
\end{align*}
where for 
the second  equality we use the fact that 
$h(\kappa,x,y):=(\kappa-k_\circ)P_{\la}^{\kappa}(x,y)$ is a smooth function in a neighborhood of $(k_\circ,\mu_1-k_\circ-1,\mu_2)$. Therefore Theorem \ref{thm:KnSa}(i) implies that \[
\widetilde{\ev}(p,\mu)=p(\mu_1-k_\circ-1,\mu_2)=0.
\]
If  $\mu$ is $k$-singular, then  Lemma \ref{lem:fk0squSingf} for  
$(a,b):=(\mu_1,\mu_2)$,
$(a',b'):=(\mu_2^\dagger,\mu_1^\dagger)$, and $p(\kappa;x,y):=P_{\la^\dagger}^\kappa(x,y)$ implies that $\widetilde{\ev}(p,\mu)=0$.  \\

\noindent\textbf{Step 2.} To determine the value of $s$, we compare the coefficient of $x^{\ul{t_1}}y^{\ul{t_2}}$ in $P_{\la^\dagger}^k$
and $p(x,y)$, where
$t_1:=\la_1^\dagger$ and $t_2:=\la_2^\dagger$. 
From \eqref{KNPsymformu} it is clear that
$x^{\ul{t_1}}y^{\ul{t_2}}$
 is the leading monomial of $P_{\la^\dagger}^k$ and therefore its coefficient is equal to 1. Furthermore, 
in the formula 
\eqref{KNPsymformu} for $P_\la^\kappa$, the monomial $x^{\ul{t_1}}y^{\ul{t_2}}$ corresponds to the term indexed by $i:=k+1$ and $j:=\la_1-\la_2-k-1$.
It is straightforward to show that the coefficient of the corresponding term in  $p(x,y)$ is equal to $r_{\la}$.
Note that by a direct calculation we have
\[
r_\la=\frac{
(-1)^{(k+|\la|)}
(\la_1-\la_2)^{\ul{k+1}}
}
{
(2k+2-\la_1+\la_2)!
(\la_1-\la_2-k-2)!
}
.
\qedhere\]
\end{proof}

%
%
%
%
%
%
%
%
%
%

For $\la\in\cPs$ set
\[
Q_{\la}^{\kappa}(x,y):=
P_{\la}^{\kappa}
-
\frac{r_\la}{\kappa-k}P_{\la^\dagger}^{\kappa}.
\]
\begin{lem}
\label{lem:formulaTlac}
For $\la\in\cPs$, 
the coefficients of $Q_{\la}^{\kappa}(x,y)$ do not have poles at 
$\kappa=k$. Furthermore, the polynomial
$Q_\la(x,y)\in\Q[x,y]$ defined by
\begin{equation}
\label{eq:Qla=lim}
Q_{\la}:=\lim_{\kappa\to k}
\left(
P_{\la}^{\kappa}-
\frac{r_{\la}}{\kappa-k}
P_{\la^\dagger}^{\kappa}
\right)
\end{equation}
satisfies \begin{equation}
\label{eq:Qla=R-la-rlaaa}
Q_{\la}(x,y)
=
R_{\la}^{(k)}(x,y)-r_\la 
\frac{\partial}{\partial\kappa}P_{\la^\dagger}^{k}(x,y).
\end{equation}
\end{lem}
\begin{proof}
By Proposition \ref{lem:sing=rlaPla} the coefficients of 
$P_\la^\kappa-\frac{r_\la}{\kappa-k}P_{\la^\dagger}^k$ do not have poles at $\kappa=k$. Furthermore, since the coefficients of $P_{\la^\dagger}^\kappa$ do not have poles at $\kappa=k$,
it follows that 
each coefficient of 
the polynomial $P_{\la^\dagger}^\kappa-P_{\la^\dagger}^k$ is of the form $(\kappa-k)\varphi(\kappa)$ where $\varphi(\kappa)\in\Q(\kappa)$ does not have a pole at  $\kappa=k$. 
Now \begin{equation}
\label{eq:inproofQla=E}
Q_\la^\kappa=P_\la^\kappa-\frac{r_\la}{\kappa-k}P_{\la^\dagger}^{k}
-
\frac{r_\la}{\kappa-k}\left(P_{\la^\dagger}^\kappa-R_{\la^\dagger}^{(k)}\right),\end{equation}
and therefore the coefficients of $Q_\la^\kappa$ do not have poles at $\kappa=k$. 
Equality \eqref{eq:Qla=R-la-rlaaa} follows from taking the limit $\kappa\to k$ in \eqref{eq:inproofQla=E}.
\end{proof}
For $\la\in\cPs$ set
\begin{equation}
\label{eq:formulaalabla}
\alpha_\la(\kappa):=
-\frac{r_\la}{\kappa-k}H_{\la^\dagger}(\kappa)
\quad\text{ and }\quad\beta_\la(\kappa):=
H_{\la}(\kappa).
\end{equation}
\begin{prp}
\label{prp:Tcheckvals}
Let $\la\in \cPs$ and let $Q_\la\in\Q[x,y]$ be defined as in~\eqref{eq:Qla=lim}. Set
\begin{equation}
\label{eq:1t2dfff}
t_1:=\beta_{\la}(k)
\quad\text{and}\quad
t_2:=\frac{\beta_\la'(k)-\alpha_\la'(k)}{4(k+1-\la_1+\la_2)},
\end{equation}
where 
$\alpha_\la(\kappa)$ and $\beta_\la(\kappa)$ are defined in 
\eqref{eq:formulaalabla}.
Then 
\[
\widetilde{\ev}(Q_\la,\mu)=t_1\delta_{\la^\dagger,\mu}+t_2\delta_{\la,\mu}
\quad\text{for all 
$\mu\in\cI$ satisfying $|\mu|\leq |\la|$.}
\]
\end{prp}

\begin{proof}
First assume that either $\mu$ is $k$-regular,
or $\mu$ is $k$-quasiregular and $\mu\neq \la^\dagger$. By Theorem 
\ref{thm:KnSa},
for $\kappa$ chosen sufficiently close (but not equal) to $k$ we have
\begin{equation}
\label{eq:Tmu-k=0}
Q_{\la}^{\kappa}(\mu_1-\kappa-1,\mu_2)=
P_{\la}^{\kappa}(\mu_1-\kappa-1,\mu_2)-
\frac{r_\la}{\kappa-k}P_{\la^\dagger}^{\kappa}
(\mu_1-\kappa-1,\mu_2)=0.
\end{equation}
By~\eqref{eq:Qla=lim} we have  $Q_{\la}(x,y)=\lim_{\kappa\to k}Q_{\la}^\kappa(x,y)$, so that  
$\widetilde{\ev}(Q_\la,\mu)=0$.

Next assume that $\mu$ is $k$-singular
 and $\mu\neq \la$.
We use Lemma \ref{lem:squPbetal}
for $p(\kappa;x,y):=Q_\la^\kappa(x,y)$, $(a,b)=(\mu_1,\mu_2)$, 
and $(a',b')=(\mu_2^\dagger,\mu_1^\dagger)$. Note that 
Theorem \ref{thm:KnSa} implies $\alpha(\kappa)=\beta(\kappa)=0$, from which it follows that $\widetilde{\ev}(Q_\la,\mu)=0$.

Next assume that 
$\mu=\la^\dagger$.
Then symmetry of $Q_\la(x,y)$ implies\begin{align*}
\widetilde{\ev}(Q_\la,\mu)&=Q_\la(\la_1^\dagger-k-1,\la_2^\dagger)=
Q_\la(\la_2^\dagger,\la_1^\dagger-k-1)\\
&=Q_\la(\la_1-k-1,\la_2)
=\lim_{\kappa\to k}Q_\la^\kappa(\la_1-\kappa-1,\la_2)=\lim_{\kappa\to k}H_\la(\kappa)=
H_\la(k).
\end{align*}
Finally, assume that 
$\mu=\la$. Then 
\[
Q_\la^\kappa(\mu_1-\kappa-1,\mu_2)=
H_\la(\kappa)
\quad\text{and}
\quad
Q_\la^\kappa(\mu_1^\dagger-\kappa-1,\mu_2^\dagger)=
-\frac{r_\la}{\kappa-k}H_{\la^\dagger}(\kappa)
.
\]
Lemma~\ref{lem:squPbetal} for $p(\kappa;x,y):=Q_\la^\kappa$, $(a,b):=(\la_1^\dagger,\la_2^\dagger)$, and $(a',b'):=(\la_2,\la_1)$ yields
$\widetilde{\ev}(Q_\la,\mu)=t_2$.
\end{proof}

We are now ready to complete the proofs of Theorems~\ref{theorem1},~\ref{theorem2}, and~\ref{theorem4}.
\subsection*{Proof of Theorem~\ref{theorem1}}
Fix $\la\in\cPr$ and set $d:=|\la|$. By Corollary  \ref{cor:ev(f,la)=tla} it suffices to show that 
\[
\widetilde{\ev}\left(
\frac{1}{H_\la(k)}R_\la^{(k)},\mu\right)=\delta_{\la,\mu}
\quad\text{ for $\mu\in\cI(d)$.}
\] Note that
$R_\la^{(k)}=P_\la^k$
. If $\mu$ is 
$k$-regular or $k$-quasiregular, 
this follows from 
taking the limit $\kappa\to k$
in Theorem~\ref{thm:KnSa}.
If $\mu$ is $k$-singular, we   
set $k_\circ:=k$, $(a,b):=(\mu_1,\mu_2)$ and $(a',b'):=(\mu^\dagger_2,\mu^\dagger_1)$ in 
Lemma~\ref{lem:squPbetal}, and note that 
Theorem~\ref{thm:KnSa} implies  $\alpha'(k)=\beta'(k)=0$.

\subsection*{Proof of Theorem~\ref{theorem2}}
By Corollary~\ref{cor:ev(f,la)=tla} 
 it suffices to prove that
\[
\widetilde{\ev}
\left(
\frac{4(\la_1-\la_2-k-1)}{H'_{\la^\dagger}(k)}R_{\la^\dagger}^{(k)}
,\mu\right)=\delta_{\la,\mu}\text{ for }\mu\in\cI,\ |\mu|\leq |\la|.
\]
The  argument is based on Lemma \ref{lem:squPbetal} and is 
similar to the proof of Theorem \ref{theorem1}.

\subsection*{Proof of Theorem \ref{theorem4}}
Set $h(\kappa;x,y):=\frac{1}{H_\la(\kappa)}P_\la^\kappa+\frac{1}{H_{\la^\dagger}(\kappa)}P_{\la^\dagger}^\kappa$, so that the right hand side of~\eqref{eq:tm3lim} is equal to 
$\lim_{\kappa\to k}h(\kappa)$. The latter limit exists because 
by~\eqref{eq:defofrlam} we have
\[
h(\kappa;x,y)=
\frac{1}{H_{\la^\dagger}(\kappa)}Q_{\la^\dagger}^\kappa(x,y)+\frac{u(\kappa)}{H_{\la^\dagger}(\kappa)}
P_{\la}^\kappa(x,y)\quad
\text{where}
\quad
u(\kappa):=\left(
\frac{H_{\la^\dagger}(\kappa)}{H_{\la}(\kappa)}
-
\frac{H_{\la^\dagger}(k)}{(\kappa-k)H'_{\la}(k)}
\right),
\]
and \[
\lim_{\kappa\to k}u(\kappa)=\frac{d}{d\kappa}\left(
\frac{H_{\la^\dagger}(\kappa)(\kappa-k)}{H_{\la}(\kappa)}\right)\Big|_{\kappa=k}.
\]
Next we set $h_1(x,y):=\lim_{\kappa\to k} h(\kappa;x,y)$.
Since $\deg  f_\la=|\la|$,
by Corollary~\ref{cor:ev(f,la)=tla}
in order to prove that $f_\la=h_1$ it suffices to verify that 
$\mathrm{ev}(h_1,\mu)=\delta_{\la,\mu}$ for all $\mu\in\cI$ such that $|\mu|\leq |\la|$. If $\mu\neq\la^\dagger$, this follows from 
Proposition~\ref{prp:Tcheckvals} and Theorem~\ref{theorem2}. If $\mu=\la^\dagger$, this follows from Lemma~\ref{lem:squPbetal} because
\[
h(\kappa;\la_1^\dagger-\kappa-1,\la^\dagger_2)=h(\kappa;\la_2,\la_1-\kappa-1)=1.
\]

\section{Proof of Theorem~\ref{theorem3}}
We begin the proof of Theorem~\ref{theorem3} by the following proposition which is a variation of Theorem~\ref{theorem4}.

\begin{prp}
\label{prp:variationThmC}
For $\la\in\cPq$, we have 
\begin{equation}
\label{eqf:HlaQla}
\displaystyle f_{\la}=\frac{1}{H_{\la^\dagger}(k)}\left(Q_{\la^\dagger}+\frac{\beta'_{\la^\dagger}(k)-\alpha'_{\la^\dagger}(k)}{H'_{\la}(k)}R_{\la}^{(k)}\right),
\end{equation}
where $Q_{\la^\dagger}$ is defined as in Lemma~\ref{lem:formulaTlac}.
\end{prp}
\begin{proof}
Note that $R_{\la}^{(k)}=P_{\la}^k$.
Since $\deg(Q_{\la^\dagger})=\deg(f_{\la^\dagger})=\deg(f_{\la})=|\la|$, from Proposition~\ref{prp:Tcheckvals}  and Corollary \ref{cor:ev(f,la)=tla} 
it follows that $Q_{\la^\dagger}=t_1f_{\la}+t_2f_{\la^\dagger}$, so that 
\[
f_{\la}=\frac{1}{t_1}\left(Q_{\la^\dagger}-t_2f_{\la^\dagger}\right),
\]
where $t_1$ and $t_2$ are defined in \eqref{eq:1t2dfff}.
The claim now follows from Theorem~\ref{theorem2}. 
\end{proof}

From now on we assume that $\la\in\cPq$. Then we can express $\la$ and $\la^\dagger$ as 
\[
\la=(d+k+1,d+\ell+1)\quad\text{and}\quad
\la^\dagger=(d+k+\ell+2,d),
\] 
where $d\geq 0$ and $0\leq \ell\leq k$. Note that $\ell=\ell_{\la^\dagger}$, where $\ell_{\la^\dagger}$ is defined in~\eqref{eq:llamb}.
\begin{prp}
\label{prp:easyformulalsing}
Suppose that
$
P_{\la}^{\kappa}=\sum \alpha_{m,n}(\kappa)x^{\underline{m}}y^{\underline{n}}
$,
where $\alpha_{m,n}(\kappa)\in\Bbbk$. 
Then
\begin{align}
\label{eq:flambda=1/i1/l+1}
f_{\la}&=
\frac{(\ell+1)!}{d!(d+\ell+1)!(k+\ell+2)!}
R_{\la^\dagger}^{(k)}+
\frac{(-1)^{\ell+1}(k+1)^{\underline{\ell+1}}}{d!(d+\ell+1)!(k+1)!\ell!}
\sum \alpha_{m,n}'(k)x^{\underline{m}}y^{\underline{n}}
\\
&+
\frac{(-1)^\ell}{d!(d+\ell+1)!(k-\ell)!\ell!}
\left(
\sum_{i=d+1}^{d+\ell+1}\frac{1}{i}-\frac{1}{\ell+1}
\right)R_{\la}^{(k)}.
\notag
\end{align}
\end{prp}

\begin{proof}
From Proposition~\ref{prp:variationThmC} and \eqref{eq:Qla=R-la-rlaaa} we obtain
\begin{equation}
\label{eq:fla-finalformula1}
f_{\la}
=
\frac{1}{H_{\la^\dagger}(k)}
\left(
R_{\la^\dagger}^{(k)}-r_{\la^\dagger}\frac{\partial P_{\la}^\kappa}{\partial \kappa}
+
\frac{\beta'_{\la^\dagger}(k)-\alpha'_{\la^\dagger}(k)}{H'_{\la}(k)}
R_{\la}^{(k)}\right).
\end{equation}
By straightforward calculations we can verify that
\[
r_{\la^\dagger}=\frac{(-1)^\ell(k+\ell+2)^{\ul{k+1}}}{(k-\ell)!\ell!},
\quad
H_{\la^\dagger}(k)=\frac{(k+\ell+2)!d!(d+\ell+1)!}{(\ell+1)!},
\quad\text{and}\quad
\alpha_{\la^\dagger}(k)=\beta_{\la^\dagger}(k)=H_{\la^\dagger}(k).
\]
Next note that for a polynomial $f(\kappa):=c\prod_i(a_i-\kappa)$ we have $f'(k)=-f(k)\sum_i\frac{1}{a_i-k}$. 
In particular
\[
\alpha_{\la^\dagger}'(k)=-H_{\la^\dagger}(k)\sum_{i=\ell+1}^d\frac{1}{i}
\quad
\text{and}\quad
\beta_{\la^\dagger}'(k)=
-H_{\la^\dagger}(k)\sum_{i=\ell+2}^{d+\ell+1}\frac{1}{i}.
\]
Finally, $H'_{\la}(k)=
\lim_{\kappa\to k}\left(\frac{H_{\la}(\kappa)}{\kappa-k}\right)=(-1)^{\ell+1}(k-\ell)!(d+\ell+1)!d!\ell!$, and the claim of the proposition follows by making  substitutions in \eqref{eq:fla-finalformula1}.
\end{proof}

\subsection*{Proof of Theorem \ref{theorem3}}
By equating the right hand sides of  
\eqref{eq:flambda=1/i1/l+1} 
and
\eqref{eq:flaformula-sing}, and then reparametrizing the summation in 
\eqref{eq:flambda=1/i1/l+1}
in terms of $a:=k-\ell-\mu_1$ and $b:=\mu_2$, 
it follows that Theorem~\ref{theorem3} is equivalent to the equation
\begin{equation}
\label{eq:alpha'=a+b+1}
\frac{\partial P_{\la}^\kappa}{\partial \kappa}
=\sum_{\mu\in\cI^*(k-\ell)}\frac{(-1)^{k-\ell+|\mu|+1}(\ell+1)!(\ell+k-\ell-\mu_1)!{k-\ell-\mu_1\choose \mu_2}(k-\ell)!}{(\mu_1-\mu_2)!(k-\ell-\mu_1)(\ell+\mu_2+1)!(k-\mu_1+\mu_2)!}P_{(\mu_1+\ell+1,\mu_2+\ell+1)}^{\kappa},
\end{equation}
where $\cI^*(k-\ell):=\cI(k-\ell)\bls\{(k-\ell,0)\}$.
Now set $N:=k-\ell$. From \eqref{KNPsymformu} or \cite[Cor. 2.3]{KSDiff}
it follows that
\[
P^\kappa_{(\mu_1+\ell+1,\mu_2+\ell_1)}=x^{\ul{\ell+1}}y^{\ul{\ell+1}}
P_\mu^\kappa(x-\ell-1,y-\ell-1).
\] 
Therefore after dividing both sides of 
\eqref{eq:alpha'=a+b+1} by $x^{\ul{\ell+1}}y^{\ul{\ell+1}}$ and making the substitution \[
(x,y)\mapsto (x+\ell+1,y+\ell+1)
,\]
we obtain that proving~\eqref{eq:alpha'=a+b+1} reduces to verifying
\begin{equation}
\label{eq:ddKapaa}
\frac{\partial}{\partial\kappa}
P_{(N,0)}^{\kappa}=
\sum_{\mu\in\cI^*(N)}\frac{(-1)^{N-\mu_1-\mu_2+1}(\ell+1)!(\ell+N-\mu_1)!{N-\mu_1\choose \mu_2}N!}{(\mu_1-\mu_2)!(N-\mu_1)(\ell+\mu_2+1)!(\ell+N-\mu_1+\mu_2)!}P_\mu^{\kappa}.
\end{equation}
To prove \eqref{eq:ddKapaa}, it suffices to verify that the coefficients of the terms 
$x^{\ul{i}}y^{\ul{j}}$ with $0\leq i+j\leq N$ on both sides are equal. 
These coefficients can be computed explicitly using the formula \eqref{KNPsymformu}. After some routine algebraic computations, it follows that the equality of the coefficients of
$x^{\ul{i}}y^{\ul{j}}$ on both sides of 
\eqref{eq:ddKapaa} is equivalent to the identity
\begin{align}
\label{eq:identityddk}
\frac{d}{d\kappa}
&\left(
\frac{\kappa^{\underline{N-i}}\,\kappa^{\underline{N-j}}}{\kappa^{\underline{N}}}
\right)
=\sum_{\mu}
\frac{(-1)^{N-\mu_1-\mu_2+1}{N-\mu_1\choose \mu_2}i^{\ul{\mu_2}}j^{\ul{\mu_2}}(N-i-j)^{\ul{N-\mu_1-\mu_2}}}{(N-\mu_1)\kappa^{\ul{\mu_1-\mu_2}}
(\kappa-\mu_1+\mu_2-1)^{{\ul{\mu_2}}}
(\kappa-N+\mu_2)^{\ul{\mu_2}}}
\kappa^{\ul{\mu_1-i}}\kappa^{\ul{\mu_1-j}},
\end{align}
where the summation is on all partitions 
 $\mu:=(\mu_1,\mu_2)\neq (N,0)$ that satisfy
\[
N\geq \mu_1+\mu_2\geq i+j
\ \text{ and }\ \mu_1\geq i\geq j\geq \mu_2.
\]
Note that 
\eqref{eq:identityddk} is a one-variable identity in a free parameter $\kappa$. By the above discussion, Theorem \ref{theorem3} follows from \eqref{eq:identityddk}. 
Note that 
 \eqref{eq:identityddk} is equivalent to 
Theorem~\ref{theorem-identity}
after the substitutions $x:=\kappa$ and 
$(p,q):=(\mu_1,\mu_2)$.
We will prove Theorem~\ref{theorem-identity} in the next section.

\section{Proof of Theorem~\ref{theorem-identity}}

In this section we prove Theorem~\ref{theorem-identity}, which completes the proof of Theorem~\ref{theorem3}. 
Set $d:=N-i$
and 
$
\psi_L(x):=
\frac{x^{\ul{d}}}{(x-N+1)\cdots (x-N+j)}
$.
Then identity \eqref{eq:huge-identity} is equivalent to the relation
\begin{equation}\label{eq:mainpsLRR}
\frac{d}{d x}\psi_L(x)=\psi_R(x),
\end{equation}
 where
\begin{equation*}
\psi_R(x):=\sum_{q=0}^{j}
\
\sum_{p=N-d+j-q}^{\min\{N-q,N-1\}}
\frac{
(-1)^{N+p+q+1}
(N-p)^{\ul{q}}\,(N-d)^{\ul{q}}\,j^{\ul{q}}\,
(d-j)^{\ul{N-p-q}}\,
x^{\ul{p+d-N}}\,x^{\ul{p-j}}(x-p+q)
}{
(N-p)\,q!\,x^{\ul{p+1}}
(x-N+q)^{\ul{q}}
}.
\end{equation*}
Our strategy is to prove a \emph{two}-variable identity that 
implies~\eqref{eq:mainpsLRR} as a special case.
For integers $q,r\geq 0$ such that $d\geq r\geq q\geq 0$, let $E(q,r)$ be the rational function in variables $x,y$ defined by
\begin{align*}
E(q,r):=
(-1)^{r+q+1}
r^{\ul{q}}
(x-y-d)^{\ul{q}}j^{\ul{q}}(d-j)^{\ul{r-q}}
\left(
\frac{
x^{\ul{d-r}}(y+r+q)
}
{
r q!(y+r+j)^{\ul{j+r}}
}
\right)
\left(
\frac{(y+r-1)^{\ul{r-q}}}{y+q}
\right).
\end{align*}
\begin{lem}
\label{lem:mys1}
Set \[
\psi_1(x,y):=\sum_{q=0}^j\sum_{r=\max\{1,q\}}^{d-j+q} E(q,r).
\] Then $\psi_R(x)=\psi_1(x,x-N)$.
\end{lem}
\begin{proof}
This is a straightforward computation. Note that $r=N-p$, where $p$ is as in the definition of $\psi_R(x)$. 
\end{proof}
\begin{lem}
\label{lem:mys2}
Set \[
\psi_2(x,y):=-\frac{x^{\ul{d}}}{\prod_{t=1}^j(y+t)}
\left(
-\sum_{t=1}^{j}
\frac{1}{y+t}
\right)
+
\frac{\partial x^{\ul{d}}}{\partial x}
\left(\frac{1}{\prod_{t=1}^j(y+t)}\right).
\] Then $\frac{d}{d x}\psi_L(x)=\psi_2(x,x-N)$. 
\end{lem}
\begin{proof}
This follows from computing $\frac{d}{d x}\psi_L(x)$ using the Leibniz rule. 
\end{proof}
Lemma~\ref{lem:mys1} and Lemma~\ref{lem:mys2} imply that 
in order to verify~\eqref{eq:mainpsLRR}, it suffices to prove that
\begin{equation}
\label{eq:identity-with-y}
\psi_1(x,y)=\psi_2(x,y).
\end{equation} 
 %
%
The rest of this section is devoted to the proof of \eqref{eq:identity-with-y}. 
Set 
\[
\check E(q,r):=(y+1)\cdots(y+j)E(q,r).
\]
Then 
\eqref{eq:identity-with-y} is equivalent to
\begin{equation}
\label{eq:simplified1}
-\left(\sum_{t=1}^j\frac{1}{y+t}
\right)
x^{\ul{d}}
+\frac{\partial x^{\ul{d}}}{\partial x}
=
\sum_{q=0}^j\sum_{r=\max\{1,q\}}^{d-j+q} \check E(q,r).
\end{equation}
Next set $\ell:=d-j$ and $s:=r-q$. 
For $q$ and $r$ in the range of indices on the right hand side of
\eqref{eq:simplified1} we have 
$0\leq s\leq \ell$ when $q\geq 1$, and  $1\leq s\leq \ell$ when $q=0$. Thus, the right hand side of~\eqref{eq:simplified1} can be written as a double sum over the indices $(q,s)\in\mathcal T$, where 
\[
\mathcal T:=\left\{(a_1,a_2)\in\Z^2\ :\ 0\leq a_1\leq j,\ 0\leq a_2\leq \ell,\ (a_1,a_2)\neq (0,0)\right\}.
\]
Now define $\delta_{s,0}:=1$ if $s=0$, and $\delta_{s,0}:=0$ if $s\geq 1$. 
After  
substituting $x$ by $x+j+\ell$ and dividing both sides of~\eqref{eq:simplified1}
by $(j+\ell)!$, it follows that~\eqref{eq:simplified1} is equivalent to the identity
\begin{align}
\label{eq:simplied2}
-\frac{(x+j+\ell)^{\ul{j+\ell}}}{(j+\ell)!}\left(\sum_{t=1}^j\frac{1}{y+t}
\right)
+
\frac{1}{(j+\ell)!}
\left(\frac{\partial}{\partial x}(x+j+\ell)^{\ul{j+\ell}}
\right)
=\sum_{s=0}^\ell
(-1)^{s+1}
F(s)
,
\end{align}
where
\[
F(s):=\sum_{q=\delta_{s,0}}^j
(y+2q+s)
{j\choose q}
{\ell\choose s}
\left(
\frac{(q+s-1)!}{(j+\ell)!}
\right)\frac{(y+j)^{\ul{j-q}}}{(y+q+s+j)^{\ul{j+1}}}
(x-y)^{\ul{q}}(x+j+\ell)^{\ul{j+\ell-q-s}}.
\]
Consequently, to complete the proof of Theorem~\ref{theorem-identity}, it suffices to verify~\eqref{eq:simplied2}. We will prove~\eqref{eq:simplied2} after the proof of Proposition~\ref{prp:F0Fss} below, which yields explicit formulas for $F(s)$.  
\begin{prp}
\label{prp:F0Fss}

Let $F(s)$ be defined as above. Then
\begin{equation}
\label{eq:conj1-id}
F(s)=\frac{(x+j+\ell)^{\ul{\ell-s}}(x+j)^{\ul{j}}}
{s(\ell-s)!(j+\ell)^{\ul{j}}}
\qquad
\text{ for $1\leq s\leq \ell$},
\end{equation}
and
\begin{equation}
\label{eq:conj2-id}
F(0)=
\frac
{(x+j+\ell)^{\ul{j+\ell}}}
{(j+\ell)!
}\sum_{t=1}^j\left(\frac{1}{y+t}
-\frac{1}{x+t}\right).
\end{equation}
\end{prp}

\begin{proof}
Set \[
H(s):=
\frac
{(\ell-s)!(j+\ell)^{\ul{j}}
}
{(x+j+\ell)^{\ul{j+\ell-s}}}
F(s),
\]
so that $H(s)=\sum_{q=\delta_{s,0}}^jE_1(q,s)$, where
\[
E_1(q,s):=\frac{1}{s!}
{j\choose q}
(y+2q+s)(q+s-1)!
\frac{(y+j)^{\ul{j-q}}}
{(y+q+s+j)^{\ul{j+1}}}
(x-y)^{\ul{q}}
\frac{(x+j+s)^{\ul{s}}}
{(x+q+s)^{\ul{q+s}}}.
\]
It suffices to prove that 
\begin{equation}
\label{eq:H(s)}
\displaystyle
H(s)=\frac{1}{s}\text{ for }s\in\Z\text{ such that }1\leq s\leq \ell
,\end{equation}
and
\begin{equation}
\label{eq:H(0)}
H(0)=
\sum_{k=1}^j
\frac{1}{y+k}
-
\sum_{k=1}^j
\frac{1}{x+k}
.\end{equation}
Our strategy is to relate $H(s)$ to Dougall's Theorem. First note that 
\begin{equation}
\label{eq:EPQ}
\frac{E_1(q+1,s)}{E_1(q,s)}=
\frac{h_1(q)}{h_2(q)}\ \ \text{for }0\leq q\leq j\ \text{ and }\ E_1(q,s)=0\ \text{ for }q>j,
\end{equation}
where
\[
h_1(q)=
\left(q+\frac12 y+\frac12 s+1\right)(q+y+s)(q+y+s)(q-j)(q+s)(q+y-x),
\]
and
\[
h_2(q)=
\left(q+\frac{1}{2}y+\frac{1}{2}s\right)(q+y+s+j+1)(q+y+1)(q+x+s+1)(q+1).
\]
Furthermore,
\begin{equation}
\label{eq:E1q}
E_1(0,s)=
\frac{
(y+j)^{\ul{j}}(x+j+s)^{\ul{s}}
}
{s{(y+s+j)^{\ul{j}}(x+s)^{\ul{s}}}}
\quad\text{ for }s\in\Z^+.
\end{equation}
We can write
\eqref{eq:E1q} as 
$E_1(0,s)=\frac{1}{s}\phi(s)$, where
\[
\phi(s):=
\frac{
(y+j)^{\ul{j}}\Gamma(x+j+s+1)\Gamma(x+1)
}{
(y+s+j)^{\ul{j}}\Gamma(x+s+1)\Gamma(x+j+1)
}
.\]
Thus we can extend $E_1(0,s)$ to a meromorphic function of $s$ for any choice of $x,y\in\C$. Note that if $x,y>0$ then $E_1(0,s)$ does not have any poles for $s\in\R^+$. 
Using \eqref{eq:EPQ} we can 
extend $E_1(q,s)$ for $1\leq q\leq j$ to a continuous function for $s\geq 0$ as long as $x,y,x-y-j>0$. 
In particular, under the same conditions on $x$ and $y$, we can extend 
$H(s)$ to a continuous function of the parameter $s\in \R^+$ by setting $H(s):=\sum_{q=0}^j E_1(q,s)$. 
From \eqref{eq:EPQ} 
it follows that 
{
\begin{equation}
\label{eq:prodsum}
H(s)=E_1(0,s)\left(1+\sum_{q=1}^j
\left(\prod_{0\leq q'\leq q-1}\frac{h_1(q')}{h_2(q')}
\right)
\right).
\end{equation}
The products that appear in the summands of~\eqref{eq:prodsum} simplify, and by comparing with~\eqref{eq:pFFq} we obtain

\begin{equation}
\label{eq:H(s)=FFFF}
H(s)=
E_1(0,s)
\left[
{_5\mathbf F_4}
\left(
\begin{matrix}
\frac{1}{2}y+\frac{1}{2}s+1\,,\,
y+s\,,\,
-j\,,\, 
s \,,\, 
y-x\\
\frac{1}{2}y+\frac{1}{2}s\,,\,
  y+s+j+1\,,\,
y+1 \,,\,
 x+s+1
\end{matrix}
;1\right)
\right]
\quad\text{for $s\in\R^+$}.
\end{equation}
Note that in the hypergeometric series on the right hand side of~\eqref{eq:H(s)=FFFF}, only the first $j+1$ terms are nonzero (because of the $-j$ appearing in the top row of parameters).
}
From Dougall's Theorem for
$
\textsf a=y+s$, $\textsf b=j$, 
$\textsf c=-s$, and 
$\textsf d=x-y$,
we obtain  
\begin{align*}
H(s)&
=E_1(0,s)
\frac{
\Gamma(y+s+j+1)
\Gamma(y+1)
\Gamma(x+s+1)
\Gamma(x+j+1)
}
{
\Gamma(y+s+1)
\Gamma(y+j+1)
\Gamma(x+1)
\Gamma(x+j+s+1)
}
\\
&
=E_1(0,s)
\frac{(x+j)^{\ul{j}}}{(y+j)^{\ul{j}}}
\frac{(y+s+j)^{\ul{j}}}{(x+s+j)^{\ul{j}}}.
\end{align*}
If $s\in\Z$ and $1\leq s\leq \ell$, then from \eqref{eq:E1q} it follows that
\[
H(s)=
\frac{1}{s}
\left(
\frac{
(y+j)^{\ul{j}}(x+j+s)^{\ul{s}}
}
{{(y+s+j)^{\ul{j}}(x+s)^{\ul{s}}}}
\right)
\frac{(x+j)^{\ul{j}}}{(y+j)^{\ul{j}}}
\frac{(y+s+j)^{\ul{j}}}{(x+s+j)^{\ul{j}}}=\frac{1}{s}.
\]
This completes the proof of~\eqref{eq:H(s)}.
For~\eqref{eq:H(0)}, 
set \[\psi_3(s):=
\frac{(x+j)^{\ul{j}}}{(y+j)^{\ul{j}}}
\frac{(y+s+j)^{\ul{j}}}{(x+s+j)^{\ul{j}}},
\] so that $H(s)=E_1(0,s)\psi_3(s)$. 
Then
\begin{align*}
H(0)=\sum_{q=1}^j
E_1(q,0)&=\lim_{s\to 0^+}
\sum_{q=1}^j
E_1(q,s)=\lim_{s\to 0^+}\left(H(s)-E_1(0,s)\right)\\
&=\lim_{s\to 0}
E_1(0,s)\left(
\psi_3(s)-1\right)=\lim_{s\to 0^+}
\phi(s)\left(\frac{1}{s}(\psi_3(s)-1)\right)=\psi_3'(0)\lim_{s\to 0^+}\phi(s).
\end{align*}
It is straightforward to check that $\lim_{s\to 0^+}\phi(s)=1$ and $\psi_3'(0)=\sum_{k=1}^j\frac{1}{y+k}-
\sum_{k=1}^j\frac{1}{x+k}.$ 
\end{proof}

We now return to the proof of~\eqref{eq:simplied2}. Using the Leibniz rule and \eqref{eq:j=0i} we have
\begin{align*}
\label{eq:ddxcal*}
\frac{\partial}{\partial x}(x+j+\ell)^{\ul{j+\ell}}
&=
\frac{\partial}{\partial x}
\left(
(x+j+\ell)^{\ul{\ell}}(x+j)^{\ul{j}}\right)\\
&=
(x+j)^{\ul{j}}\frac{\partial}{\partial x}
(x+j+\ell)^{\ul{\ell}}
+
(x+j+\ell)^{\ul{\ell}}
\frac{\partial}{\partial x}
(x+j)^{\ul{j}}
\\
&=(x+j)^{\ul{j}}
\left(
\sum_{t=1}^\ell(-1)^{t+1}\frac{\ell^{\ul{t}}}{t}(x+j+\ell)^{\ul{\ell-t}}
\right)+
(x+j+\ell)^{\ul{j+\ell}}
\left(
\sum_{t=1}^j\frac{1}{x+t}\right).
\end{align*}
Identity~\eqref{eq:simplied2} follows from substituting the latter formula  in its left hand side, and rewriting its right hand side using Proposition~\ref{prp:F0Fss}.


\section{Capelli operators in Deligne's Category $\Rep(O_t)$}
\label{sec:Deligne}

In this section, we define the categorical Capelli operators 
$\bsm D_{t,\la}$ and prove Theorems~\ref{thm:A'}--\ref{thm:C'}.
We begin by defining general categorical analogues of the algebras 
$\sP(V)$ and $\sPD(V)$. 
Let $\mathsf C$ be a Karoubian $\mathbb F$-linear symmetric monoidal category, where $\mathbb F$ is a field of characteristic zero.
Given an object 
$\sfX$ of $\mathsf C$, set 
$\sfP_\sfX^d:=\sS^d(\sfX)$ for $d\geq 0$ and 
$\sfP_\sfX:=\bigoplus_{d\geq 0}\sfP_\sfX^d$, where  we consider $\sfP_\sfX^d$ as an object of the inductive completion of $\mathsf C$.
 Then
$\sfP_\sfX$ is a commutative algebra object when equipped
with the multiplication morphism $\bsm\mu_\sfX:\sfP_\sfX\otimes \sfP_\sfX\to \sfP_\sfX$ that is induced from the monoidal structure of $\mathsf C$.
If $\sfX$ is left rigid and $\sfX^*$ denotes the left dual
of $\sfX$, then we set \[
\sfPD_\sfX:=\sfP_\sfX\otimes \sfP_{\sfX^*}\cong \bigoplus_{p,q\geq 0} \sS^p(\sfX)\otimes \sS^q(\sfX^*).
\]
For $q\geq p\geq 0$ the evaluation morphism
$\bsm \epsilon_{\sS^p(\sfX)}^{}:\sS^p(\sfX^*)\otimes \sS^p(\sfX)\to\yeksf$ yields a morphism 
\[
\mathbf{tr}_{p,q}:\sS^p(\sfX^*)\otimes\sS^q(\sfX)\to\sS^{q-p}(\sfX),
\]
and we set \[
\bsm\gamma_{p,q}:\sfP_\sfX\otimes \sS^p(\sfX^*)\otimes \sS^q(\sfX)\to\sfP_\sfX
\ \ ,\ \ 
\bsm \gamma_{p,q}:=
\bsm\mu_\sfX\circ (1\otimes\mathbf{tr}_{p,q}).
\]
For $p>q\geq 0$ we set  
$\bsm \gamma_{p,q}:=0$. Then 
$\bsm\gamma:=\oplus_{p,q\geq 0}\bsm\gamma_{p,q}$ is a morphism $\bsm\gamma:\sfPD_\sfX\otimes \sfP_\sfX\to \sfP_\sfX$. Moreover, there exists a unique morphism 
$
\tilde{\bsm\mu}:\sfPD_\sfX\otimes\sfPD_\sfX\to\sfPD_\sfX
$
satisfying  
\[
\bsm\gamma(\tilde{\bsm\mu}\otimes 1)=\bsm\gamma (1\otimes\bsm\gamma).
\] Thus $\sfPD_\sfX$ is an associative algebra object and $\sfP_\sfX$ is a $\sfPD_\sfX$-module 
in the
inductive completion
of~$\mathsf C$. The ``order'' filtration of $\sfPD_\sfX$ is given by setting 
\[
\sfPD_\sfX^i:=\sfP_\sfX\otimes \sS^i(\sfX^*)
\ \text{ for }i\geq 0.
\]
There is also a  $\mathbb Z$-grading on $\sfPD_\sfX$ given by 
\[
\sfPD_{\sfX,i}:=\bigoplus_{p-q=i}\sS^p(\sfX)\otimes \sS^q(\sfX^*),
\]
so that 
$\sfPD_\sfX\cong\oplus_{i\in\Z}\sfPD_{\sfX,i}$.
Note that $\sfPD_{\sfX,0}$ is a subalgebra object of $\sfPD_{\sfX}$. 
If $\bsm\iota_\sfX:\yeksf\to\sfX\otimes \sfX^*$ is the co-evaluation of $\sfX$ then clearly $\bsm\iota_X\in\Hom(\yeksf,\sfPD_0)$. 

Next suppose that there exists an isomorphism $\sfX\xrightarrow{\bsm\beta}\sfX^*$, and set 
\begin{equation}
\label{eq:oemgass}
\bsm \omega_\sfX:=(1\otimes \bsm\beta^{-1})\bsm\iota_\sfX
\quad\text{and}\quad 
 \bsm \omega_\sfX^*:=(\bsm\beta\otimes 1)\bsm\iota_\sfX.
 \end{equation} It is straightforward to verify that $\bsm\omega_\sfX\in\Hom(\yeksf,\sfPD_2)$ and 
 $\bsm\omega_\sfX^*\in\Hom(\yeksf,\sfPD_{-2})$.

We now return to the Deligne category 
$\Rep(O_t)$.
Recall that $\Rep(O_t)$ is the  Karoubian $\C$-linear rigid symmetric monoidal category generated by the self-dual  object $\V_t$ of categorical dimension $t\in\C$. 
We denote the identity object
of $\Rep(O_t)$
 by $\yeksf$ and the braiding 
of $\Rep(O_t)$ 
 by \[
 \bsm\sigma: \sfM\otimes \sfN\to \sfN\otimes \sfM.
 \]
Since $\V_t$ is self-dual, we have evaluation and 
co-evaluation morphisms 
\[
\bsm\epsilon: \V_t\otimes \V_t\to\yeksf\quad\text{and}\quad\bsm\iota:\yeksf\to \V_t\otimes \V_t,
\] that satisfy the usual duality axioms. Furthermore, these morphisms satisfy the relations 
\[
\bsm\sigma\bsm \iota=\bsm\iota\quad,\quad
\bsm \epsilon\bsm\sigma=\bsm\epsilon\quad,\quad 
\bsm\epsilon\bsm\iota=t
.\]
By definition, for  $d\geq 0$ the $\C$-algebra $\mathrm{End}(\V_t^{\otimes d})$ is generated by the morphisms
$1^{\otimes(i-1)}\otimes\bsm\sigma\otimes 1^{\otimes (d-i-1)}$ and $\bsm\iota\bsm\epsilon\otimes 1^{\otimes (d-2)}$.
The category  $\Rep(O_t)$ 
satisfies the following properties
(see~\cite{DeligneT,DeligneTensorielle,LehrerZhang}). 
\begin{prp}\label{category}
The following statements hold in 
the category $\Rep(O_t)$.
  \begin{itemize}
  \item[\rm(i)] For $d\geq 0$, the algebra $\End(\V_t^{\otimes d})$ is isomorphic to the Brauer algebra $Br_d(t)$.
    \item[\rm(ii)]  Every indecomposable object of $\Rep(O_t)$ is isomorphic to
      the image of a primitive idempotent in $\End(\V_t^{\otimes d})$ for some $d\geq 0$.
      
    \item[\rm(iii)] 
    $\Hom(\V_t^{\otimes p}, \V_t^{\otimes q})=\{0\}$ if $p-q$ is odd. If $p-q$ is even, then 
    $\Hom(\V_t^{\otimes p}, \V_t^{\otimes q})$     
 is generated    as a 
$(Br_q(t),Br_p(t))$-bimodule    
 by $\bsm\epsilon^{\otimes \frac{p-q}{2}}\otimes 1^{\otimes q}$  if $p>q$, and by 
    $\bsm\iota^{\otimes\frac{q-p}{2}}\otimes 1^{\otimes p}$ if $q>p$. 
    
    \item[\rm(iv)]  If $t\notin\mathbb Z$, then $\Rep(O_t)$ is an abelian semisimple tensor category and in particular $Br_d(t)$ is a semisimple algebra.
    \item[\rm(vi)] If $t\in\mathbb Z$ and 
    $p,q\geq \Z^{\geq 0}$ such that $p-2q=t$, then there exists a symmetric monoidal full
      functor 
      $\mathrm{F}_{p|q}:\Rep(O_t)\to\Rep(\g{osp}(p|2q))$ such that 
      $\mathrm{F}_{p|q}(\V_t)=\C^{p|2q}$, where $\C^{p|2q}$ is the defining representation of $\g{osp}(p|2q)$.
    \end{itemize}
  \end{prp}

Our next goal is to define 
categorical analogues of invariant differential operators, and in particular the Euler and Casimir operators.  
To this end,  
we 
set 
\[
\sfA_t:=\Hom\left(\yeksf,\sfPD_{\V_t,0}\right)
\quad\text{and}\quad
\sfB_t:=\Hom\left(\yeksf,\sfPD_{\V_t}\right).
\]
Then $\sfA_t$ and $\sfB_t$ are algebras with the products defined by 
\[
\bsm\alpha_1\otimes \bsm\alpha_2\mapsto
\tilde{\bsm\mu}\circ (\bsm\alpha_1\otimes \bsm\alpha_2)\circ\bsm\iota_{\circ},
\]
where $\bsm\iota_{\circ}:\yeksf\to \yeksf\otimes \yeksf$ is the co-evaluation of $\yeksf$.
One can interpret $\sfB_t$ as
the algebra of $O_t$-invariant differential operators acting on $\sfP_{\V_t}$. Similarly, $\sfA_t$ can be interpreted as the algebra of $\mathit{GO}_t$-invariant differential operators on $\sfP_{\V_t}$.

The morphism $\bsm\gamma:\sfPD_{\V_t}\otimes\sfP_{\V_t}\to\sfP_{\V_t}$ induces 
homomorphisms of associative algebras
\[
\gamma_{\sfA_t}:{\sfA_t}\to\End(\sfP_{\V_t})\quad\text{and}\quad\gamma_{\sfB_t}:\sfB_t\to\End(\sfP_{\V_t}).
\]
Then $\bsm E_t:=\gamma_{\sfA_t}(\bsm\iota)$  acts by the scalar $d$ on $\sfP^d_{\V_t}$.
Set $\boldsymbol\Delta_t:=\frac{1}{2}\gamma_{\sfB_t}^{}\left(\bsm\omega_{\V_t}^{}\right)$ and $\boldsymbol\Theta_t:=\frac{1}{2}\gamma_{\sfB_t}^{}\left(\bsm\omega_{\V_t}^*\right)$,
where $\bsm\omega_{\V_t}^{}$ and $\bsm\omega_{\V_t}^*$ are defined as in  \eqref{eq:oemgass}, with $\bsm\beta:=1_{\V_t}$. 
It is straightforward to verify the relations
\begin{equation}\label{sl2}
[\mathbf E_t,\bsm\Delta_t]=-2\bsm\Delta_t,\ [\bsm E_t,\bsm\Theta_t]=2\bsm \Theta_t,\ [\bsm\Delta_t,\bsm\Theta_t]=\bsm E_t+\frac{t}{2}.
\end{equation}
Now set
\begin{equation}\label{eq:Casimir}
  \bsm C_t:=
  (\bsm E_t+\frac{t}{2})^2-2\bsm \Theta_t\bsm\Delta_t
  -2\bsm\Delta_t\bsm \Theta_t-\frac{t^2}{4}+t=\bsm E_t^2+t\bsm E_t-4\bsm \Theta_t\bsm\Delta_t-2\bsm E_t.
  \end{equation}
  One can check that $\bsm C_t$ is indeed the Casimir element for the Lie algebra object $\g g_t\simeq\Lambda^2(\V_t)$. 
  
\begin{prp}\label{AB} Let $t\in\C$ and let $\sfA_t$, $\sfB_t$, 
$\gamma_{\sfA_t}$, and $\gamma_{\sfB_t}$ be as above. Then the following statements hold.
  \begin{itemize}
  \item[\rm (i)] 
  $\gamma_{\sfB_t}(\sfB_t)$ is generated by 
  $\bsm\Delta_t$ and $\bsm\Theta_t$.
    \item[\rm (ii)] $\gamma_{\sfB_t}(\sfB_t)$ is isomorphic to the universal enveloping algebra $U(\g{sl}_2)$.
    \item[\rm (iii)] $\gamma_{\sfA_t}(\sfA_t)$ is generated by $\bsm C_t$ and $\bsm E_t$.
\end{itemize}
\end{prp}
\begin{proof} 

(i)
Fix $b\in\gamma_{\sfB_t}(\sfB_t)$ and choose $d\in\N$ sufficiently large such that the restriction of  $b$ on 
$\sfP^{\leq d}_{\V_t}:=
  \bigoplus_{p\leq d}\sfP^{p}_{\V_t}$ uniquely determines $b$ among elements of $\gamma_{\sfB_t}(\sfB_t)$. 
  Since projections from 
$\sfP^{\leq d}_{\V_t}$ onto $\sfP^{p}_{\V_t}$ for $0\leq p\leq d$ can be expressed as polynomials in $\bsm E_t$, we can write $b$ as 
\[
b=\sum_{0\leq p,q\leq d} f_q(\bsm E_t)bg_p(\bsm E_t),
\] where 
$f_i,g_i\in\C[x]$. Each summand $f_q(\bsm E_t)bg_p(\bsm E_t)$ can be identified with an element of $\Hom(\sfP^{p}_{\V_t},\sfP^{q}_{\V_t})$ for some $p,q\geq 0$. To complete the proof of (i), it suffices to express element of $\Hom(\sfP^{p}_{\V_t},\sfP^{q}_{\V_t})$ in terms of 
$\bsm\Delta_t$ and $\bsm\Theta_t$.
   To prove the latter claim, first assume $p=q$. Recall that the algebra
$\End(\V_t^{\otimes p})$ is generated by the symmetric group $S_p$ and the morphism $\bsm\iota\bsm\epsilon\otimes 1^{\otimes (p-2)}$. Since 
$\sfP_{\V_t}^p=\sS^p(\V_t)$ is a direct summand of $\V_t^{\otimes p}$, the canonical restriction $\End(\V_t^{\otimes p})\to \End(\sfP_{\V_t}^p)$ is a surjection. 
But the action of $S_p$
  on $\sS^p(\V_t)$ is trivial, hence $\End(\sfP_{\V_t}^p)=\End(\sS^p(\V_t))$ is generated by $\bsm\Theta_t\bsm\Delta_t$. 
Next assume that $p\neq q$.   Then by Proposition~\ref{category}(iii), for $q>p$ the homomorphism 
  \[
  \C[\bsm\Theta_t\bsm\Delta_t]
  \bsm\Theta_t^{\frac{q-p}{2}}
  \C[\bsm\Theta_t\bsm\Delta_t]\to\Hom(\sS^p(\V_t), \sS^q(\V_t))
  \] is surjective. Similarly,  for  $p>q$ the homomorphism
  \[
  \C[\bsm\Theta_t\bsm\Delta_t]\bsm\Delta_t^{\frac{p-q}{2}}\C[\bsm\Theta_t\bsm\Delta_t]\to\Hom(\sS^p(V_t), \sS^q(V_t))
  \]
  is surjective.\\[1mm]
  \noindent (ii) Since 
$\bsm\Delta_t,\bsm E_t+\frac{t}{2},-\bsm\Theta_t$ form a standard $\mathfrak{sl}_2$-triple, we obtain a surjection $
U(\g{sl}_2)\to \gamma_{\sfB_t}(\sfB_t)
$. Next we prove that the latter homomorphism is injective.  

First, we assume that $t\notin 2\mathbb Z$. For every simple object $\sfX$ of $\Rep(O_t)$
the space
$\sfM_\sfX:=\Hom(\sfX,\sfP_{\V_t})$ is a $\gamma_{\sfB_t}(\sfB_t)$-module and hence  a $U(\g{sl}_2)$-module.
It suffices to show that $\sfM:=\oplus_\sfX \sfM_\sfX$ is a faithful $U(\g{sl}_2)$-module, where the direct sum is taken over isomorphism classes of simple objects of $\Rep(O_t)$.
Note that each $\sfM_\sfX$ is a weight module with weights in $\Z^{\geq 0}+\frac{t}{2}$. From the theory of Verma modules for  $\g{sl}_2$ it follows  that if $d$ is such that 
$
\dim\Hom(\sfX,\sfP_{\V_t}^{d-2})<\dim\Hom(\sfX,\sfP_{\V_t}^{d})
$, then  
$\sfM$ contains a Verma module
with lowest weight $d+\frac{t}{2}$ as a subrepresentation. Next we show that for all but finitely many $d\geq 0$, the latter inequality holds for some $\sfX$. 
Indeed since $\bsm\Theta_t$ induces a monomorphism $\sfP_{\V_t}^{d-2}\to \sfP_{\V_t}^d$, it suffices to show that $\sfP_{\V_t}^{d-2}$ and $\sfP_{\V_t}^d$ are not isomorphic objects. The latter follows from comparing the categorical dimensions, which is given by the formula
\[
\dim \sfP^d_{\V_t}=\frac{t(t+1)\dots (t+d-1)}{d!}.
\]
Hence $\sfM$ contains $\g{sl}_2$-submodules which are Verma modules with lowest weights $d+\frac{t}{2}$ for all but finitely many $d\in\mathbb N$. 
The intersections of the annihilators of these Verma modules is the trivial ideal of $U(\g{sl}_2)$ (see \cite[Sec. 8.4]{Dixmier}), hence $\sfM$ is  a faithful $U(\g{sl}_2)$-module.

If  $t\in 2\mathbb Z$ the result follows from the analogous result for $\g{osp}(2m|2n)$ with $2m-2n=t$, where $m,n\in\N$ (see for instance~\cite{ShermanPhD}) using the functor
$\mathrm F_{m|n}$ 
defined in Proposition~\ref{category}(vi).\\[1mm]
\noindent (iii) Note that $\gamma_{\sfA_t}(\sfA_t)$ is the centralizer of $\bsm E_t$ inside $\gamma_{\sfB_t}(\sfB_t)$. Thus (ii) implies that (iii) is equivalent to the well-known fact that the centralizer of the Cartan subalgebra in  $U(\mathfrak{sl}_2)$ is generated by the Cartan subalgebra and the Casimir operator.
  \end{proof}

\begin{lem}\label{CD} Let $d\in\Z^{\geq 0}$ and  let $\bsm C_{t,d}$ 
be the image of $\bsm C_t$ in $\End(\sfP_{\V_t}^d)$.
Let $p_t^d(x)\in\C[x]$ be the minimal degree monic polynomial such that $p^d_t(\bsm C_{t,d})=0$.
  Then
  \[
  p^d_t(x)=\prod_{\tiny\begin{array}{l}0\leq a\leq d\\
  d\equiv a\ \mathrm{mod}\ 2\end{array}} (x-a(a+t-2))
  .\]
\end{lem}

\begin{proof} For $d\leq 1$ the statement is trivial since $\bsm C_{t,0}=0$ and $\bsm C_{t,1}=(t-1)1_{\V_t}$. We will prove the statement by induction on $d$.

 First we assume
 that $t\notin\Z$. We claim that \[\sfP_{\V_t}^d\cong \sfP_{\V_t}^{d-2}\oplus\ker \bsm\Delta_t\big|_{\sfP^d_{\V_t}}.\] Indeed, from representation theory of $\g{sl}_2$ (see the proof of Proposition~\ref{AB}(ii)) it follows that $\bsm\Delta_t$ is surjective. Semisimplicity of $\Rep(O_t)$ implies the claim. 
  
  By \eqref{eq:Casimir}, the operator $\bsm C_{t,d}$ acts on $\ker \bsm\Delta_t$ 
  by the scalar  $d(d+t-2)$.
  Since $d(d+t-2)$ is not a root of  
  $p^{d-2}_t(x)$, we have   $p^d_t(x)=(x-d(d+t-2))p^{d-2}_t(x)$. The statement now follows by induction.

  Next assume that $t\in \Z$.  Choose positive integers $a,b$ such that $a-2b=t$.
  Then from Proposition~\ref{prp:decompositionofS(V)}
  it follows that
  $p^r_t(x)$ is the minimal polynomial for $\mathrm F_{a|b}(\bsm C_{t,d})$,
where $\mathrm F_{a|b}$ is the functor given in Proposition~\ref{category}(iv).
The homomorphism $\mathrm F_{a|b}: \End(\sfP^d)\to \End ( \sP^d(V))$ is surjective since $\mathrm F_{a|b}$ is full. On the other hand,
$ \End(\sfP^d)$ is spanned by $\left\{\bsm \Theta_t^p\bsm\Delta_t^p\,:\,0\leq p\leq \lfloor{\frac{d}{2}}\rfloor\right\}$ (see the proof of Proposition~\ref{AB}(i)). Hence
\[
\dim \End(\sfP_{\V_t}^d)\leq
1+\Big\lfloor{\frac{d}{2}}\Big\rfloor
=\dim\End ( \sP^d(V)),
\]
and thus $\mathrm F_{a|b}$ is an isomorphism. Consequently,  the minimal polynomials of $\mathrm F_{a|b}(\bsm C_{t,d})$ and  $\bsm C_{t,d}$ are identical.
The statement now follows from the decomposition of $\sP^d(V)$ as a $\g{gosp}(a|2b)$-module (see
Proposition~\ref{prp:decompositionofS(V)} and~\cite[Sec. 10]{ShermanPhD}).
\end{proof}

\begin{rmk}
\label{rem:CD} If $t\notin 2\mathbb Z^{\leq 0}$ then $p^d_t(x)$ does not have multiple roots. If $t\in   2\mathbb Z^{\leq 0}$ then the multiplicity of
each root of  $p^d_t(x)$ is at most $2$.
\end{rmk}

\begin{lem} 
\label{lem:Wuis}
Let 
$p_t^d(x)$ be as in Lemma~\ref{CD}. Let $u_i$, $1\leq i\leq e$, be the distinct roots of $p^d_t(x)$ with corresponding multiplicities $m_{u_i}\in\{1,2\}$. 
Set $\mathsf W_{u_i}:=\ker\left( (\bsm C_{t,d}-u_i)^{m_{u_i}}\right)$. Then
$\sfP_{\V_t}^d\cong\bigoplus_{i=1}^e 
\mathsf W_{u_i}$ and 
 every $\mathsf W_{u_i}$ is an indecomposable object of $\Rep(O_t)$.
\end{lem}
\begin{proof} 
The proof of the decomposition 
$\sfP^d_{\V_t}\cong\bigoplus_{i=1}^e 
\mathsf W_{u_i}$
is similar to that of the Primary Decomposition Theorem in linear algebra.  Set
$q_i(x):=\frac{p_t^d(x)}{(x-u_i)^{m_{u_i}}}$
and choose
$g_i,h_i\in\C[x]$  such that $
q_i(x)g_i(x)+(x-u_i)^{m_{u_i}}h_i(x)=1$.
Then the morphisms 
$\bsm\pi_i:=q_i(\bsm C_{t,d})g_i(\bsm C_{t,d})$ are the projections onto the $\mathsf W_{u_i}$.  

Next we show that each $\mathsf W_{u_i}$ is indecomposable. 
Recall that $\bsm C_{t,d}$ generates the algebra 
$\End(\sfP^d_{\V_t})$ (see the proof of Proposition~\ref{AB}(i)). Since $\mathsf W_{u_i}$ is a direct summand, it follows that $\End(\mathsf W_{u_i})$ is also generated by the restriction of $\bsm C_{t,d}$.
 Consequently, $\End(\mathsf W_{u_i})\cong\C[x]/\lag(x-u_i)^{m_{u_i}}\rag $, hence 
$\mathsf W_{u_i}$ is indecomposable.
  \end{proof}

Recall 
from~\eqref{eq:kbarrr} that
$\ul k:=-\frac{t}{2}$.
For $\la\in\cI$ and $t\in\C$ set
\[
\bsm c_\lambda(t):=(\lambda_1-\lambda_2)(\lambda_1-\lambda_2+t-2)
=(\la_2-\la_1)(\la_2-\la_1+2\ul k+2). 
\]
Now let $d\in\Z^{\geq 0}$, and let $p_t^d(x)$ be as in Lemma~\ref{CD}. Given $\la\in\cI$ such that $|\la|=d$, we  denote the multiplicity of the root $\bsm c_\la(t)$ of 
$p_t^d(x)$ by $m_\la$. 
Lemma~\ref{lem:Wuis} immediately implies the following corollary, which is the categorical analogue of~\eqref{eq:Pdco}.  Recall that 
$\cI'_{\ul k}$ is defined as in~\eqref{eq:kbarrr}.

\begin{cor}\label{Ddecomposition} 
For $\la\in\cI$  
set $\V_{t,\la}:=\ker\left((\bsm C_{t,d}-\bsm c_\la(t))^{m_\la}\right)$ where $d:=|\la|$. Then the following statements hold.   
\begin{itemize}
\item[\rm(i)] If $t\not\in 2\Z^{\leq 0}$ then  $\sfP^d_{\V_t}\cong
\bigoplus_{\lambda\in\cI,|\lambda|=d}\V_{t,\la}$.
\item[\rm (ii)] If $t\in 2\mathbb Z^{\leq 0}$ then $\sfP^d_{\V_t}\cong
\bigoplus_{\lambda\in _d^{}\cI'_{\ul k}}\V_{t,\la}$.
\end{itemize}
\end{cor}  
We are now going to define the \emph{eigenvalue polynomial} $\bsm f_{\bsm D}$ for an element $\bsm D\in\sf A_t$.  
From now on, for $\bsm D\in\sfA_t$ we denote the restriction of $\bsm \gamma_{\sfA_t}(\bsm D)$ to $\V_{t,\mu}$ by $\bsm D\big|_{\V_{t,\mu}}$. 
\begin{prp}
\label{prp:Delieig}Let $\bsm D\in\sfA_t$. Set $\mathcal S:=\cI$ if $t\not\in2\Z^{\leq0}$ and $\mathcal S:=\cI_{\ul k}'$ otherwise. Then there exists a unique symmetric polynomial $\bsm f_D(x,y)$ such that for every $\la\in\mathcal S$, we have
\[
\bsm D\big|_{\V_{t,\la}}=
\bsm f_{\bsm D}(\la_1-\ul k-1,\la_2)\cdot 1_{\V_{t,\la}}+\bsm \eta_\la,
\]  
where $\bsm \eta_\la\in\End(\V_{t,\la})$ satisfies $\bsm \eta_\la^2=0$. 
\end{prp}
\begin{proof}
Existence follows from Corollary~\ref{Ddecomposition} and Proposition~\ref{AB}(iii), and the argument is similar to the proof of Proposition~\ref{prp:CandEgen}. Uniqueness follows from the fact that $\mathcal P$ is Zariski dense in $\C^2$. Note that when $t\in2\Z^{\leq 0}$, the coincidence relation~\eqref{eq:coinc} implies that two symmetric polynomials that agree on $\mathcal S$ also agree on $\cI$. 
\end{proof} 

\begin{dfn}
\label{dfn-evalpolp}
Let $\bsm D\in\sfA_t$. The \emph{eigenvalue polynomial} of $\bsm D$ is the polynomial
$\bsm f_{\bsm D}(x,y)\in\C[x,y]$ that is given in Proposition~\ref{prp:Delieig}.
\end{dfn}

{
\subsection*{The construction of the Capelli operators in $\Rep(O_t)$}
Our next task is to define the Capelli operators $\{\bsm D_{t,\la}\}_{\la\in\cI}$. 
\begin{itemize}
\item If $t\notin 2\mathbb Z^{\leq 0}$ then for all $\la\in\cI$ we define $\bsm D_{t,\lambda}\in\sfA_t$ as the element corresponding to the co-evaluation morphism
\begin{equation}
\label{eq:CategCap}
\yeksf\xrightarrow{\ \bsm\epsilon_{\V_{t,\la}}^{}\ } \V_{t,\la}\otimes \V_{t,\la}^*.
\end{equation}
\item 
If $t\in 2\mathbb Z_{\leq 0}$ 
and $\lambda$ 
is $\ul k$-regular or $\ul k$-quasiregular, we define $\bsm D_{t,\lambda}$ as 
in~\eqref{eq:CategCap}. 

\item
If $t\in 2\mathbb Z_{\leq 0}$ and 
$\lambda$ is $\ul k$-singular, 
we
define 
$\bsm D_{t,\la}$ as the element of $\sfA_t$ corresponding to the morphism
\begin{equation}
\label{eq:CategCap2}
\yeksf\xrightarrow{\,\bsm C_{t,|\la|}-\bsm c_\lambda(t)\,} \V_{t,\la^\dagger}\otimes \V_{t,\lambda^\dagger}^*,
\end{equation}
Using the fact that $(\bsm C_{t,|\la|}
-\bsm c_\lambda(t))
\in \End(\V_{t,\la^\dagger})\cong \Hom(1,\V_{t,\la^\dagger}\otimes \V_{t,\la^\dagger}^*)$.
Here and in the rest of the paper $\la^\dagger$ is defined as in 
Remark~\ref{daggerinv}, but with $k$ replaced by $\ul k$. 
Lemma~\ref{lem:Wuis} implies that 
$\bsm C_{t,|\la|}-\bsm c_\lambda(t)$ is a nilpotent element of order two in $\End(\V_{t,\la^\dagger})$.

\end{itemize}
}

For $d\geq 0$ let $\sfJ_t^{d}$  denote the annihilator of $\sfP^{\leq d}_{\V_t}:=\bigoplus_{p\leq d}\sfP^p_{\V_t}$ in $\sfA_t$.
Since $\sfA_t$ is commutative, $\sfJ^d$ is a 
two-sided ideal of $\sfA_t$. Moreover, we have a decomposition
\[
\sfA_t=\sfA_t^d\oplus\sfJ_t^d,
\]
where $\sfA_t^d:=\sfPD_{\V_t}^d\cap\sfA_t$.
Let \[
\bsm\pi_{t,d}: \sfA_t\to\sfA_t^d
\] denote the projection with kernel $\sfJ_t^d$.
Our next task is to show that as $t$ varies, the projections $\bsm\pi_{t,d}$ deform with polynomials coefficients. 
From Proposition~\ref{AB} it follows that $\bsm\gamma_{\sfA_t}$ is an injection. Thus, from now on we identify $\bsm C_t$ and $\bsm E_t$ with their images under $\bsm \gamma_{\sfA_t}$.

\begin{lem}
\label{lem:polynompitd}
For $i,j,d\geq 0$, there exist polynomials
$\phi_{i,j,d,i',j'}\in\C[x]$ such that 
\[
\bsm \pi_{t,d}(\bsm C_t^i\bsm E_t^j)=\sum_{2i'+j'\leq d}
\phi_{i,j,d,i',j'}(t)\bsm C_t^{i'}\bsm E_t^{j'}\ \text{ for all }t\in\C.
\]
\end{lem}
\begin{proof}
We will describe a recursive procedure for finding the $\phi_{i,j,d,i',j'}$ with the desired properties. \\[1mm]
\noindent\textbf{Step 1}. Using one-variable interpolation, for every $j\geq 0$ we can find scalars $a_p$ for $0\leq p\leq d$  such that $\bsm E_t^j-\sum_{p=0}^da_p\bsm E_t^p\in\sfJ_t^d$. This proves the statement for the special case $i=0$. \\[1mm]
\noindent\textbf{Step 2}. We show that for every $i\geq 0$ there exist polynomials $\psi_{p,q}(t)$ for $1\leq p\leq N_i$ and $0\leq q\leq \lfloor\frac{d}{2}\rfloor$, where $N_i\in\N$, such that the element $\bsm L\in\sfA_t$ defined by 
\begin{equation}
\label{eq:Lqqq}
\bsm L:=\sum_{p=0}^{\lfloor\frac{d}{2}\rfloor}
\sum_{q=1}^{N_i}
\psi_{p,q}(t)\bsm C_t^p
\bsm E_t^q,
\end{equation}
satisfies $\bsm C_t^i-\bsm L\in\sfJ_t^d$.
Indeed by~\eqref{sl2} and~\eqref{eq:Casimir} we can write 
$\bsm C_t^i$ as a linear combination of monomials of the form $\bsm E_t^p\bsm\Theta_t^q\bsm\Delta_t^q$ with coefficients that are polynomial in $t$. Furthermore we can discard the monomials for which $q>\lfloor \frac{d}{2}\rfloor$, because 
${\bsm\Delta}_t^q\big|_{\sS^d(\V_t)}=0$. 
Next we use~\eqref{sl2} and~\eqref{eq:Casimir} again to rewrite
${\bsm\Theta}_t^q	\bsm\Delta_t^q$ first in terms of powers of $\bsm \Theta_t\bsm\Delta_t$ and then in terms of powers of $\bsm C_t$. The latter process  will only add extra powers of $\bsm E_t$ and coefficients that are polynomial in $t$. This completes the proof of existence of $\bsm L$.\\[1mm]
\noindent\textbf{Step 3}. Fix a pair $(i,j)$ of exponents. Assume that the statement of the lemma holds for all $\bsm C_t^{i'}\bsm E_t^{j'}$ such that either $i'<i$, or $i'=i$ and $2i'+j'<2i+j$. We now verify the lemma  for
$\bsm C_t^{i}\bsm E_t^j$. If $2i+j\leq d$ there is nothing to prove, and therefore we assume that $2i+j\geq d+1$. 
If $j=0$ then using Step 2 we can reduce the problem to monomials of the form 
$\bsm C_t^{i'}\bsm E_t^{j'}$ where $i'<i$. 
If $j>0$ then set
$
\bsm L':=
\bsm C_t^i\prod_{r=2i}^{2i+j-1}(\bsm E_t-r)
$. Note that 
$\bsm C_t-\bsm L'$ is a linear combination of monomials 
$\bsm C_t^i\bsm E_t^{j'}$ satsifying $2i+j'<2i+j$.
Furthermore, 
by Step 2 the restriction of $\bsm C_t^i$ to 
$\bigoplus_{p=0}^{2i-1}\sS^p(\V_t)$ is equal to a linear combination  of  monomials $\bsm C_t^p\bsm E_t^q$ where $p\leq i-1$, with polynomial coefficients. 
It follows that the restriction of 
$\bsm L'$ to $\bigoplus_{p=0}^{d}\sS^p(\V_t)$
is also equal to a linear combination
 of monomials $\bsm C_t^p\bsm E_t^q$ where $p\leq i-1$, with polynomial coefficients. 
\end{proof}

The next lemma is the categorical incarnation  of 
Lemma~\ref{lem:lamuvalRep}.

\begin{lem}\label{vanishing} Let
$\la\in\cI$ and set  $d:=|\la|$.
\begin{itemize}
\item[\rm (i)]  If $t\not\in2\Z^{\leq0}$, 
then 
$\bsm D_{t,\la}$ is the unique element of 
$\sfA_t^d$ such that
$\bsm D_{t,\la}\big|_{\V_{t,\la}}=1$
and   $\bsm D_{t,\la}\big|_{\V_{t,\mu}}=0$ for all $\mu\in\cI$ satisfying
  $|\mu|\leq |\la|$ and $\mu\neq\lambda$.

\item[\rm (ii)]
If
$t\in2\Z^{\leq0}$ and $\la$ is $\ul k$-regular,
then 
$\bsm D_{t,\la}$ is the unique element of 
$\sfA_t^d$ such that
$\bsm D_{t,\la}\big|_{\V_{t,\la}}=1$
and   $\bsm D_{t,\la}\big|_{\V_{t,\mu}}=0$ for all $\mu\in\cI'_{\ul k}$ satisfying
  $|\mu|\leq |\la|$ and $\mu\neq\lambda$.

\item[\rm (iii)]
If $t\in2\Z^{\leq0}$ and $\la$ is 
$\ul k$-singular, 
then $\bsm D_{t,\la}$ is the unique element of 
$\sfA_t^d$ such that
\[
\bsm D_{t,\lambda}\big|_{\V_{t,\lambda^\dagger}}=
\bsm C_{t,d}-\bsm c_\lambda(t)
,\]
and 
 $\bsm D_{t,\lambda}\big|_{\V_{t,\mu}}=0$ for all $\mu\in\cI'_{\ul k}$ satisfying $|\mu|\leq |\la|$ 
and $\mu\neq\lambda^\dagger$.

\item[\rm (iv)] 
If $t\in2\Z^{\leq0}$ and $\la$ is $\ul k$-quasiregular,  then 
$\bsm D_{t,\la}$ is the unique element of 
$\sfA_t^d$ such that
\[
\bsm D_{t,\la}\big|_{\V_{t,\la}}=1,
\]
and   $\bsm D_{t,\la}\big|_{\V_{t,\mu}}=0$ for all $\mu\in\cI'_{\ul k}$ satisfying
  $|\mu|\leq |\la|$ 
and $\mu\neq\lambda$.

\end{itemize}
\end{lem}

\begin{proof} The stated properties of $\bsm D_{t,\la}$ are straightforward from the definition. Uniqueness  follows from the fact that any element of $\sfA_t^d$ is uniquely determined by its restriction to a morphism of~$\sfP^{\leq d}_{\V_t}$.
  \end{proof}

Let $\la\in\cI$ and 
set $d:=|\la|$. 
For $s\in\C$ such that $s\not\in 2\Z^{\leq0}$
we define ${\bsm L}_{s,t,\lambda}\in\sfA_t$ by
\begin{equation}
\label{eq:Dtildsla}
{\bsm L}_{s,t,\lambda}:=\frac
{
\prod_{i=0}^{d-1}(\bsm E_t-i)\prod_{|\nu|=d, \nu\neq\lambda}
\left(\bsm C_{t}-\bsm c_\nu(s)\right)}{d!\prod_{|\nu|=d, \nu\neq\lambda}(\bsm c_\lambda(s)-\bsm c_\nu(s))}.
\end{equation}
We remark that
$\bsm L_{s,t,\la}$ is well-defined because
the factors $(\bsm c_\la(s)-\bsm c_\nu(s))$ in the denominator of \eqref{eq:Dtildsla} vanish 
only if $s\in2\Z^{\leq 0}$ and  $\la$ and $\nu$ are a pair of
$k_s$-quasiregular and $k_s$-singular partitions, where $k_s:=-\frac{s}{2}$.
We can now expand the right hand side of~\eqref{eq:Dtildsla} and express $\bsm L_{s,t,\la}$ as 
\begin{equation}
\label{eq:Llasssas}
{\bsm L}_{s,t,\la}=\sum_{i,j\geq 0}\eta_{\la,i,j}(s)\bsm C_t^i\bsm E_t^j,
\end{equation}
where the $\eta_{\la,i,j}$ are rational functions of $s$. Note that the $\eta_{\la,i,j}$
are independent of $t$ and do not have  poles in $\C$ 
outside the set $2\Z^{\leq 0}$.

\begin{dfn}

\label{defDlsst}
For $s\in\C$ such that $s\not\in 2\Z^{\leq 0}$,
we define $\bsm D_{s,t,\la}\in\sfA_t$ as follows.
\begin{itemize}
\item[(i)] If either $t\not\in 2\Z^{\leq 0}$,  
or $t\in 2\Z^{\leq 0}$ and $\la$ is $\ul k$-regular, then we set
$\bsm D_{s,t,\la}:={\bsm L}_{s,t,\lambda}$.
\item[(ii)] 
If $t\in2\Z^{\leq 0}$ and $\la$ is 
$\ul k$-singular, then 
we set 
$\bsm D_{s,t,\la}:=(\bsm c_{\la^\dagger}(s)-\bsm c_{\la}(s)){\bsm L}_{s,t,\lambda^\dagger}$.

\item[(iii)]
 If $t\in2\Z^{\leq 0}$ and $\la$ is 
$\ul k$-quasiregular, then 
we set 
$\bsm D_{s,t,\la}:=
{\bsm L}_{s,t,\lambda}+{\bsm L}_{s,t,\lambda^\dagger}$.

\end{itemize}

\end{dfn}
Using~\eqref{eq:Llasssas} we can express $\bsm D_{s,t,\lambda}$ as 
\begin{equation}
\label{eq:coefDstla}
\bsm D_{s,t,\lambda}=\sum_{i,j}\eta_{i,j}(s)\bsm C_t^i\bsm E_t^j,
\end{equation} where $\eta_{i,j}(s)$ is equal to $\eta_{\la,i,j}(s)$ or
$(\bsm c_{\la^\dagger}(s)-\bsm c_{\la}(s))\eta_{\la^\dagger,i,j}(s)$ or $\eta_{\la,i,j}(s)+\eta_{\la^\dagger,i,j}(s)$ 
in cases (i), (ii), and (iii) of Definition~\ref{defDlsst}, respectively.

For $\ul k$-quasiregular $\la$  we define
\[
\tilde\phi_{1,\la}(s):=\prod_{|\nu|=|\la|,\nu\neq\la,\la^\dagger}(\bsm c_\la(s)-\bsm c_\nu(s))^{-1}
\quad\text{and}\quad
\tilde\phi_{2,\la}(s):=\prod_{|\nu|=|\la|,\nu\neq\la,\la^\dagger}(\bsm c_{\la^\dagger}(s)-\bsm c_\nu(s))^{-1}.
\]
The next proposition is a key step in the proofs of Theorems~\ref{thm:A'}--\ref{thm:C'}.

\begin{prp}
\label{prp:Dstlanoplies}
The rational functions 
$\eta_{i,j}(s)$
in~\eqref{eq:coefDstla} do not have any poles at $s=t$.
\end{prp}

\begin{proof}
For 
$\bsm D_{s,t,\la}$ as in
Definition~\ref{defDlsst}(i)-(ii) this follows
from the fact that for $\la,\nu\in\cI$ such that $|\la|=|\nu|$, we have $\bsm c_\la(s)=\bsm c_\nu(s)$ if and  only if $s\in2\Z^{\leq 0}$ and  $\la$ and $\nu$ are a pair of
$(-\frac{s}{2})$-quasiregular and $(-\frac{s}{2})$-singular partitions.
For $\bsm D_{s,t,\la}$ as in Definition~\ref{defDlsst}(iii)
note that
\begin{align}
\label{eq:expDstla}
\bsm D_{s,t,\la}
&=
\prod_{|\nu|=d, \nu\neq\lambda,\la^\dagger}
\left(\bsm C_{t}-\bsm c_\nu(s)\right)
\frac
{
\prod_{i=0}^{d-1}(\bsm E_t-i)
}{d!}
\bsm D',
\end{align}
where $d:=|\la|$ and  
\[
\bsm D':=\left(
\tilde\phi_{1,\la}(s)
\frac{\bsm C_t-\bsm c_{\la^\dagger}(s)}
{\bsm c_\la(s)-\bsm c_{\la^\dagger}(s)
}
-\tilde\phi_{2,\la}(s)\frac{\bsm C_t-\bsm c_{\la}(s)}
{\bsm c_\la(s)-\bsm c_{\la^\dagger}(s)}
\right)
.
\]
Then $\bsm D'=\gamma_0(s)+\gamma_1(s)\bsm C_t$ where 
\[
\gamma_0(s)=-\frac{\tilde \phi_{1,\la}(s)\bsm c_{\la^\dagger}(s)-\tilde\phi_{2,\la}(s)\bsm c_\la(s)}
{\bsm c_\la(s)-\bsm c_{\la^\dagger}(s)}
\quad
\text{and}
\quad
\gamma_1(s)=\frac{\tilde \phi_{1,\la}(s)-\tilde\phi_{2,\la}(s)}
{\bsm c_\la(s)-\bsm c_{\la^\dagger}(s)}
.\]
From 
the remark about vanishing of the differences $(\bsm c_\la(s)-\bsm c_\nu(s))$ it follows that 
$\tilde\phi_{1,\la}(s)$
and
$\tilde\phi_{2,\la}(s)$
do not have poles at $s=t$. 
Furthermore, from $\bsm c_\la(t)=\bsm c_{\la^\dagger}(t)$ it follows that
$\tilde\phi_{1,\la}(t)=\tilde\phi_{2,\la}(t)$. This implies that $\gamma_0(s)$ and $\gamma_1(s)$ do not have poles at $s=t$. Hence the coefficients $\eta_{i,j}(s)$ of $\bsm D_{s,t,\la}$ are also regular at $s=t$. 
\end{proof}
Because of Proposition~\ref{prp:Dstlanoplies},
for $t\in 2\Z^{\leq 0}$ and $\la\in\cI$ we can define 
\begin{equation}
\label{DstDtt}
\bsm D_{t,t,\la}:=\sum_{i,j}\eta_{i,j}(t)\bsm C_t^i\bsm E_t^j=\lim_{s\to t}\bsm D_{s,t,\la}.
\end{equation}

\begin{prp}
\label{prp:DtlaandDtt}
$\bsm D_{t,\la}=\bsm\pi_{t,d}(\bsm D_{t,t,\la})$.
\end{prp}

\begin{proof}
It suffices to check that $\bsm D_{t,t,\la}$ satisfies the  
vanishing properties given in Lemma~\ref{vanishing}. 
If $t\not\in 2\Z^{\leq 0}$ then $\mathsf{Rep}(O_t)$ is semisimple, and in particular the $\V_{t,\nu}$ are simple objects. It is then straightforward to check the action of $\bsm D_{t,t,\la}=\bsm L_{t,t,\la}$ on each
 $\V_{t,\nu}$ 
using~\eqref{eq:Dtildsla}.
If $t\in 2\Z^{\leq 0}$ then from
Lemma~\ref{lem:polynompitd} and Proposition~\ref{prp:Dstlanoplies} it follows that 
\[
\bsm\pi_{t,d}(\bsm D_{t,t,\la})=\lim_{s\to t}\bsm\pi_{t,d}
(\bsm D_{s,t,\la}),
\] and again we can compute the action of 
$\bsm D_{s,t,\la}$ using~\eqref{eq:Dtildsla}.
The argument is by a case by case consideration, and we will only give the details for
the most difficult case, i.e., when $\la$ is $\ul k$-quasiregular. 
In this case $\bsm D_{s,t,\la}=\bsm L_{s,t,\la}+\bsm L_{s,t,\la^\dagger}$ for $s$ sufficiently close but not equal to $t$. Now choose $\nu\in\cI'_{\ul k}$. If $|\nu|<|\la|$ then both $\bsm L_{s,t,\la}$ and 
$\bsm L_{s,t,\la^\dagger}$
vanish on $\V_{t,\nu}$ because they contain the factor $(\bsm E_t-|\nu|)$. This implies that $\bsm D_{s,t,\la}\big|_{\V_{t,\nu}}=0$, hence 
$\bsm D_{t,t,\la}\big|_{\V_{t,\nu}}=0$. 
Next assume that $|\nu|=|\la|$ and set $d:=|\la|$. If $\nu$ is $\ul k$-regular, then 
$\V_{t,\nu}$ is a simple object 
and  both 
$\bsm L_{s,t,\la}$ and 
$\bsm L_{s,t,\la^\dagger}$ contain the factor
$(\bsm C_t-\bsm c_\nu(s))$, which acts on 
$\V_{t,\nu}$ by $(\bsm c_\nu(t)-\bsm c_\nu(s))1_{\V_{t,\nu}}$. Since $\lim_{s\to t}(\bsm c_\nu(t)-\bsm c_\nu(s))=0$, we obtain $\bsm\pi_{t,d}(\bsm D_{t,t,\la})\big|_{\V_{t,\nu}}=0$.
If 
$\nu$~is $\ul k$-quasiregular and $\nu\neq\la$, then
from~\eqref{eq:expDstla} it follows that
\[
\bsm D_{s,t,\la}=
(\bsm C_t-\bsm c_\nu(s))
(\bsm C_t-\bsm c_{\nu^\dagger}(s))
\bsm D
\] for some $\bsm D\in \sfA_t$ of the form
$\bsm D=\sum_{i,j}\psi_{i,j}(s)\bsm C_t^i\bsm E_t^j$, 
where the $\psi_{i,j}$ are rational functions without poles at $s=t$. 
Now set $\bsm N:=(\bsm C_t-\bsm c_{\nu^\dagger}(t))\big|_{\V_{t,\nu}}$. Then 
\[
(\bsm C_t-\bsm c_\nu(s))
(\bsm C_t-\bsm c_{\nu^\dagger}(s))\big|_{\V_{t,\nu}}
=
(\bsm c_{\nu^\dagger}(t)-\bsm c_{\nu}(s)+\bsm N)
(\bsm c_{\nu^\dagger}(t)-\bsm c_{\nu^\dagger}(s)+\bsm N)
.
\] 
As $\bsm N^2=0$ and 
$
\lim_{s\to t}(\bsm c_{\nu^\dagger}(t)-\bsm c_{\nu^\dagger}(s))=
\lim_{s\to t}(\bsm c_{\nu^\dagger}(t)-\bsm c_{\nu}(s))=0
$, we obtain $\lim_{s\to t}\bsm D_{s,t,\la}\big|_{\V_{t,\nu}}=0$. 
Finally, if  
$\nu=\la$ then 
from~\eqref{eq:expDstla} it follows that
$\bsm D_{s,t,\la}\big|_{\V_{t,\nu}}
=\bsm D^{(1)}\bsm D^{(2)}$ for
\[
\bsm D^{(1)}
:=\prod_{|\eta|=d,\eta\neq\la,\la^\dagger}
(\bsm c_\la(t)-\bsm c_\eta(s)+\bsm N)
\]
and 
\[\bsm D^{(2)}
:=\left(
\frac{\tilde\phi_{1,\la}(s)}{\bsm c_\la(s)-\bsm c_{\la^\dagger}(s)}
(\bsm c_\la(t)-\bsm c_{\la^\dagger}(s)+\bsm N)
-
\frac{\tilde\phi_{2,\la}(s)}{\bsm c_\la(s)-\bsm c_{\la^\dagger}(s)}
(\bsm c_\la(t)-\bsm c_{\la}(s)+\bsm N)
\right),
\]
Since $\bsm N^2=0$,
we have
$\bsm D^{(i)}\big|_{\V_{t,\nu}}=\gamma_0^{(i)}(s)+\gamma_1^{(i)}(s)\bsm N$
for $i\in\{1,2\}$, so that   
\begin{equation}
\label{eq:idempo}
\bsm D_{s,t,\la}\big|_{\V_{t,\nu}}=\gamma_0(s)+\gamma_1(s)\bsm N,
\end{equation}
where $\gamma_0(s)
=\gamma_0^{(1)}(s)
\gamma_0^{(2)}(s)$ 
and 
$\gamma_1(s)=\gamma_0^{(1)}(s)\gamma_1^{(2)}(s)+\gamma_1^{(1)}(s)\gamma_0^{(2)}(s)$.
To complete the proof we need to verify that 
 $\lim_{s\to t}\gamma_0(s)=1$ and $\lim_{s\to t}\gamma_1(s)=0$. 

To prove
 $\lim_{s\to t}\gamma_0(s)=1$
first note that 
$\lim_{s\to t}\gamma_0^{(1)}(s)=\tilde\phi_{1,\la}(t)^{-1}$.
Furthermore, 
\[
\gamma_0^{(2)}(s)=\tilde\phi_{1,\la}(s)
\frac
{\bsm c_\la(t)-\bsm c_{\la^\dagger}(s)}
{\bsm c_\la(s)-\bsm c_{\la^\dagger}(s)}
-
\tilde\phi_{2,\la}(s)
\frac
{\bsm c_\la(t)-\bsm c_{\la}(s)}
{\bsm c_\la(s)-\bsm c_{\la^\dagger}(s)},
\]
and from $\tilde\phi_{1,\la}(t)=\tilde\phi_{2,\la}(t)$ it follows that 
$\lim_{s\to t}
\gamma_0^{(2)}(s)=
\tilde\phi_{1,\la}(t)$.

To prove  $\lim_{s\to t}\gamma_1(s)=0$, note that 
\[
\gamma_1(s)=
\tilde\phi_{1,\la}(s)
\tilde\phi_{2,\la}(s)\left(
\prod_{|\nu|=d\,,\,\nu\neq \la,\la^\dagger}(\bsm c_\la(t)-\bsm c_\nu(s))
\right)
\left(
\tilde\phi_3(s)+
\tilde\phi_4(s)\right)
\]
where $\tilde\phi_3(s):=\frac{\tilde\phi_{2,\la}(s)^{-1}-\tilde \phi_{1,\la}(s)^{-1}}{\bsm c_\la(s)-\bsm c_{\la^\dagger}(s)}
$ and 
\begin{align*}
\tilde\phi_4(s)&:
=
\left(\sum_{|\nu|=d\,,\,\nu\neq \la,\la^\dagger}
\frac{1}{\bsm c_\la(t)-\bsm c_\nu(s)}
\right)
\left(
\frac{
\tilde\phi_{2,\la}(s)^{-1}
(\bsm c_\la(t)-\bsm c_{\la^\dagger}(s))
-
\tilde\phi_{1,\la}(s)^{-1}
(\bsm c_\la(t)-\bsm c_{\la}(s))
}
{\bsm c_\la(s)-\bsm c_{\la^\dagger}(s)}
\right).
\end{align*}
From $\tilde\phi_{1,\la}(t)=\tilde\phi_{2,\la}(t)$
it follows that
\begin{equation}
\label{eq:phise3tild}
\lim_{s\to t}\tilde\phi_4(s)=
\tilde\phi_{1,\la}(t)^{-1}
\sum_{|\nu|=d\,,\,\nu\neq \la,\la^\dagger}
\frac{1}{\bsm c_\la(t)-\bsm c_\nu(t)}
.
\end{equation}
Furthermore
$\bsm c_\la(s)-\bsm c_{\la^\dagger}(s)=(s-t)\left(2(\la_1-\la_2)+t-2\right)$, so that 
\begin{align*}
\lim_{s\to 0}
\tilde\phi_3(s)&=
\lim_{s\to t}
\frac{
(
\tilde\phi_{2,\la}(s)^{-1}-\tilde\phi_{2,\la}(t)^{-1}
)-
(
\tilde \phi_{1,\la}(s)^{-1}-\tilde\phi_{1,\la}^{-1}(t))}{\bsm c_\la(s)-\bsm c_{\la^\dagger}(s)}\\
&=
\frac{1}{2(\la_1-\la_2)+(t-2)}
\left(
\frac{d}{ds}\tilde\phi_{2,\la}(s)^{-1}\big|_{s=t}
-
\frac{d}{ds}\tilde\phi_{1,\la}(s)^{-1}\big|_{s=t}\right)
.\end{align*}
Now  
\begin{equation}
\label{eq:ddsofphit1}
\frac{d}{ds}\tilde\phi_{1,\la}(s)^{-1}\big|_{s=t}=
\tilde\phi_{1,\la}(t)^{-1}
\sum_{|\nu|=d\,,\,\nu\neq\la,\la^\dagger}
\frac{(\la_1-\la_2)-(\nu_1-\nu_2)}{\bsm c_\la(t)-\bsm c_\nu(t)}
\end{equation}
and
\begin{equation}
\label{eq:ddsofphit2}
\frac{d}{ds}\tilde\phi_{2,\la}(s)^{-1}\big|_{s=t}=
\tilde\phi_{2,\la}(t)^{-1}
\sum_{|\nu|=d\,,\,\nu\neq\la,\la^\dagger}
\frac{(-\la_1+\la_2-t+2)-(\nu_1-\nu_2)}{\bsm c_\la(t)-\bsm c_\nu(t)}.
\end{equation}
From~\eqref{eq:phise3tild},~\eqref{eq:ddsofphit1} and~\eqref{eq:ddsofphit2} it follows that $\lim_{s\to t}(\tilde\phi_3(s)+\tilde\phi_4(s))=0$, hence 
$\lim_{s\to t}\gamma_1(s)=0$.
\end{proof}

\begin{rmk}
There is  a more conceptual argument for proving
$\lim_{s\to t}\gamma_1(s)=0$ in~\eqref{eq:idempo} as follows. The construction of $\mathsf{Rep}(O_t)$ is valid over the field $\C(\xi)$ of rational functions in a parameter $\xi$, yielding a Karoubian rigid symmetric monoidal category 
generated by a self-dual object $\V_\xi$ of dimension $\xi$. Let us denote the latter category by $\mathsf{Rep}(O_\xi)$. The algebra $\sfA_t$ and the operators $\bsm C_t$, $\bsm E_t$, and $\bsm D_{t,\la}$ have 
counterparts $\sfA_\xi$, $\bsm C_\xi$, $\bsm E_\xi$, and $\bsm D_{\xi,\la}$
 in the inductive completion of $\mathsf{Rep}(O_\xi)$. For $t\in\C$, let $\mathscr O_t\sseq \C(\xi)$ denote the local ring of rational functions without a pole at $\xi=t$, and let 
 $\oline \sfA_\xi\sseq \sfA_\xi$ be the $\mathscr O_t$-subalgebra of $\sfA_\xi$ generated by $\bsm C_\xi$ and $\bsm E_\xi$. Further, let  $\mathrm{ev}_{\xi=t}:\oline \sfA_\xi\to \sfA_t$ be the ring homomorphism obtained by naturally extending  
$\mathrm{ev}_{\xi=t}(\bsm C_\xi):=\bsm C_t$ and $\mathrm{ev}_{\xi=t}(\bsm E_\xi):=\bsm E_t$. One can show that $\mathsf{Rep}(O_\xi)$ is semisimple, and it follows that   the restriction of 
$\bsm D_{\xi,\la}$ to $\sfP^{\leq d}_{\V_{\xi}}$ is an idempotent morphism. One can then use the fact that $\mathrm{ev}_{\xi=t}$ is a ring homomorphism to prove that  $\mathrm{ev}_{\xi=t}(\bsm D_{\xi,\la})$ is an idempotent when restricted to $\sfP^{\leq d}_{\V_t}$, and therefore it does not have a nonzero nilpotent part. 
 \end{rmk}

\begin{lem}
\label{lem:limitarg}
Assume that $t\in 2\Z^{\leq 0}$. Let $U_t\sseq \C$ be an open set such that $U_t\cap \Z=\{t\}$. For $s\in U_t$ let $\bsm L_s\in\sfA_s$ be defined by 
\begin{equation}\label{eq:Ls==}
\bsm L_s:=\sum_{i,j=0}^p \psi_{i,j}(s)\bsm C_s^i\bsm E_s^j,
\end{equation}
where the $\psi_{i,j}$ are rational functions without poles in $U_t$. Let
$\bsm f_{\bsm L_s}$ be defined as in Definition~\ref{dfn-evalpolp} for $s\in U_t$. 
Then $\bsm f_{\bsm L_t}=\lim_{s\to t}\bsm f_{\bsm L_s}$ as elements of $\C[x,y]$.

\end{lem}

\begin{proof}
Let $s\in U_t\bls \{t\}$. Then the category $\Rep(O_s)$ is semisimple and by~\eqref{eq:Ls==} we have
\[
\textstyle
\bsm f_{\bsm L_s}(\mu_1+\frac{s}{2}-1,\mu_2)
=
\sum_{i,j=0}^p
\psi_{i,j}(s)\bsm c_\mu(s)^i(\mu_1+\mu_2)^j
\quad\text{for $\mu\in \cI$.}
\] Since $\cI$ is Zariski dense in $\C^2$, 
it follows that 
\[
\textstyle
\bsm f_{\bsm L_s}(x,y)
=
\sum_{i,j=0}^p
\psi_{i,j}(s)
((x-y)^2-(\frac{s}{2}-1)^2)
(x+y)^j.
\]
In particular, the coefficients of $\bsm f_{\bsm L_s}(x,y)$ are rational functions without poles in $U_t$. Thus, the limit $\lim_{s\to t}\bsm f_{\bsm L_s}$ exists.

The action of $\bsm L_t$ on $\V_{t,\mu}$, 
where $\mu\in\cI'_{\ul k}$, 
is equal to 
$\sum_{i,j=0}^p \psi_{i,j}(t)(\bsm c_\mu(t)+\bsm N)^i(\mu_1+\mu_2)^j$, where $\bsm N$ is the nilpotent part of  $\bsm C_t\big|_{\V_{t,\mu}}$ (recall that $\bsm N^2=0$). Thus 
\[
\bsm f_{\bsm L_t}\left(\mu_1-\ul k-1,\mu_2\right)=\sum_{i,j=0}^p \psi_{i,j}(t)\bsm c_\mu(t)^i(\mu_1+\mu_2)^j.
\]
Consequently for all $\mu\in\cI'_{\ul k}$,
\begin{equation}
\label{eq:fbsljd[}
\bsm f_{\bsm L_t}\left(\mu_1-\ul k-1,\mu_2\right)=
\lim_{s\to t}
\bsm f_{\bsm L_s}\left(\mu_1+\frac{s}{2}-1,\mu_2\right)
=
\lim_{s\to t}
\bsm f_{\bsm L_s}\left(\mu_1-\ul k-1,\mu_2\right), 
\end{equation}
where for the second equality we use the fact that the coefficients of $\bsm f_{\bsm L_s}(x,y)$ do not have poles in $U_t$. But then~\eqref{eq:fbsljd[} also holds for all $\mu\in\cI$
since 
both $\bsm f_{\bsm L_t}$ and  
$\lim_{s\to t}\bsm f_{\bsm L_s}$ are symmetric polynomials. Since $\cI$ is Zariski dense in $\C^2$, we obtain $\bsm f_{\bsm L_t}=\lim_{s\to t}\bsm f_{\bsm L_s}$.
%
%
\end{proof}

We are now ready to prove Theorems~\ref{thm:A'}--\ref{thm:C'}. Recall that $\bsm f_\la:=\bsm f_{\bsm D_{t,\la}}$, where the right hand side is defined as in Definition~\ref{dfn-evalpolp}. Since $\bsm D_{t,\la}\in\sfA_t^d$, Lemma~\ref{lem:polynompitd}
implies that $\deg \bsm f_\la\leq |\la|$.

\subsection*{Proof of Theorems~\ref{thm:A'}--\ref{thm:C'}}

First assume that $t\not\in 2\Z^{\leq 0}$. Then 
by Lemma~\ref{vanishing}(i) 
 the polynomial 
$\bsm f_{\lambda}$ satisfies vanishing conditions analogous to the hypotheses  of 
Theorem~\ref{thm:KnSa}, and the claim follows. 
Next assume that $t\in 2\Z^{\leq 0}$. Our strategy is to reduce this case to the case $t\not\in 2\Z^{\leq 0}$. 
Let $\eta_{i,j}(s)$ be as in~\eqref{eq:coefDstla}.
Set $U_t:=\{t\}\cup (\C\bls \Z)$. By Lemma~\ref{lem:polynompitd}, 
\[
\bsm\pi_{t,d}(\bsm D_{s,t,\la})=
\sum_{i,j,i',j'}\eta_{i,j}(s)
\phi_{i,j,d,i',j'}(t)
\bsm C_s^i\bsm E_s^j
\quad\text{for $s\in U_t$}.
\]
Using 
Proposition~\ref{prp:DtlaandDtt} and 
Proposition~\ref{prp:Dstlanoplies}
 we obtain 
\begin{align*}
\bsm D_{t,\la}&=
\bsm\pi_{t,d}(\bsm D_{t,t,\la})
=\lim_{s\to t}\bsm\pi_{t,d}(\bsm D_{s,t,\la})\\
&=\lim_{s\to t}\sum_{i,j,i',j'}\eta_{i,j}(s)
\phi_{i,j,d,i',j'}(t)
\bsm C_t^i\bsm E_t^j
=\lim_{s\to t}\sum_{i,j,i',j'}\eta_{i,j}(s)
\phi_{i,j,d,i',j'}(s)
\bsm C_t^i\bsm E_t^j,
\end{align*}
where in the last step we use 
Proposition~\ref{prp:Dstlanoplies}.
Set \[
\bsm L_s:=
\sum_{i,j,i',j'}
\eta_{i,j}(s)
\phi_{i,j,d,i',j'}(s)
\bsm C_s^i\bsm E_s^j,
\] so that $\bsm L_t=\bsm D_{t,\la}$.
Next fix $s\in U_t\bls \{t\}$. We use the special case of 
Theorem~\ref{thm:A'} that was proved above to compute $\bsm f_{\bsm L_s}$. 
If $\la$ is $\ul k$-regular, then $\bsm L_s=\bsm\pi_{s,d}(\bsm L_{s,s,\la})$ and thus $
\bsm f_{\bsm L_s}=\frac{1}{H_\lambda\left(-\frac{s}{2}\right)}
  P^{-\frac{s}{2}}_\lambda$.
If $\la$ is $\ul k$-singular, then $\bsm L_s=
(\bsm c_{\la^\dagger}(s)-\bsm c_{\la}(s))\bsm\pi_{s,d}(\bsm L_{s,s,\la^\dagger})$ and thus 
\[
\bsm f_{\bsm L_s}=\frac{\bsm c_{\la^\dagger}(s)-\bsm c_{\la}(s)}{H_\lambda\left(-\frac{s}{2}\right)}
  P^{-\frac{s}{2}}_{\lambda^\dagger}.
  \]
Finally, if $\la$ is $\ul k$-quasiregular, then 
 $\bsm L_s=\bsm\pi_{s,d}(\bsm L_{s,s,\la}+\bsm L_{s,s,\la^\dagger})$ and thus 
 \[
 \bsm f_{\bsm L_s}=\frac{1}{H_\lambda\left(-\frac{s}{2}\right)}
  P^{-\frac{s}{2}}_\lambda
  +
\frac{1}{H_{\lambda^\dagger}
\left(-\frac{s}{2}\right)}
  P^{-\frac{s}{2}}_{\lambda^\dagger}  
  .
  \]
Now Lemma~\ref{lem:limitarg} implies Theorems~\ref{thm:A'}--\ref{thm:C'}.

\bibliographystyle{plain}
\bibliography{SingularCapelli}

\end{document}